\documentclass[reqno]{amsart}
\usepackage{a4wide}
\usepackage{graphicx}
\usepackage{enumerate}
\usepackage{cases}
\usepackage{bbm}
\usepackage{mathtools}
\usepackage[utopia]{mathdesign}
\usepackage[utf8]{inputenc}
\usepackage[T1]{fontenc}

\usepackage{tikz}
\usetikzlibrary{backgrounds}
\usetikzlibrary{patterns,fadings}
\usetikzlibrary{arrows,decorations.pathmorphing}
\usetikzlibrary{calc}
\definecolor{light-gray}{gray}{0.95}
\usepackage{float}
\usepackage[colorlinks=true,linkcolor=blue,citecolor=magenta]{hyperref}

\newcommand{\pfrac}[2]{\genfrac{}{}{}{1}{#1}{#2}}

\newtheorem{theorem}{Theorem}[section]
\newtheorem{lemma}[theorem]{Lemma}
\newtheorem{proposition}[theorem]{Proposition}
\newtheorem{corollary}[theorem]{Corollary}
\newtheorem{remark}[theorem]{Remark}
\newtheorem{definition}{Definition}[section]

\newcommand{\ioe}{\iota_\eps}
\newcommand{\iog}{\iota_\tau^{\textrm{s}}}
\newcommand{\ioz}{\iota_\zeta}

\numberwithin{equation}{section}

\newcommand{\subm}{\delta_{\eta^n}}

\newcommand{\radonNinv}{\frac{\textrm{\textbf{d}}\bb P_{\delta_{\eta^n}}}{\textrm{\textbf{d}}\bb P^H_{\delta_{\eta^n}}}}
\newcommand{\radonN}{\frac{\textrm{\textbf{d}}\bb P^H_{\delta_{\eta^n}}}{\textrm{\textbf{d}}\bb P_{\delta_{\eta^n}}}}

\newcommand{\one}{\textbf{1}}

\newcommand{\FDir}{\mathcal{F}_{\text{\rm Dir}}}
\newcommand{\FNeu}{\mathcal{F}_{\text{\rm Neu}}}

\newcommand{\mc}[1]{{\mathcal #1}}

\newcommand{\bb}[1]{{\mathbb #1}}
\newcommand{\bs}[1]{{\boldsymbol #1}}
\newcommand{\<}{\langle}
\renewcommand{\>}{\rangle}

\renewcommand{\epsilon}{\varepsilon}

\newcommand{\p}{\partial}
\newcommand{\eps}{\varepsilon}
\def\centerarc[#1](#2)(#3:#4:#5){\draw[#1] ($(#2)+({#5*cos(#3)},{#5*sin(#3)})$) arc (#3:#4:#5);}

\newcommand{\dradon}{{\rm{\bf{d}}}}

\newcommand{\g}{g_{\alpha,\beta}}

\newcommand{\at}[2]{\genfrac{}{}{0pt}{}{#1}{#2}}

\newcommand{\Ddiscreto}{\mc D_{\Omega_n}}
\newcommand{\DM}{\mc D_{\mc M}}
\newcommand{\DMO}{\mc D_{\mc M_0}}

\newcommand{\Sob}{L^2(0,T;\mc H^1)}
\newcommand{\I}{{\bf{I}}^\theta_T}
\newcommand{\C}{{\rm \bf C}_\theta}

\let\oldtocsection=\tocsection
\let\oldtocsubsection=\tocsubsection
\let\oldtocsubsubsection=\tocsubsubsection
\renewcommand{\tocsection}[2]{\hspace{0em}\oldtocsection{#1}{#2}}
\renewcommand{\tocsubsection}[2]{\hspace{1em}\oldtocsubsection{#1}{#2}}
\renewcommand{\tocsubsubsection}[2]{\hspace{2em}\oldtocsubsubsection{#1}{#2}}
\DeclareRobustCommand{\SkipTocEntry}[5]{}

\title[Large Deviations for the SSEP  with  slow boundary]{Large Deviations for the SSEP\\with  slow boundary: the non-critical case}



\author{Tertuliano Franco}
\address{UFBA\\
 Instituto de Matem\'atica, Campus de Ondina, Av. Adhemar de Barros, S/N. CEP 40170-110\\
Salvador, Brasil}
\email{tertu@ufba.br}
\thanks{}

\author{Patr\'{\i}cia Gon\c{c}alves}
\address{\noindent Center for Mathematical Analysis,  Geometry and Dynamical Systems,
Instituto Superior T\'ecnico, Universidade de Lisboa, 
Av. Rovisco Pais, 1049-001 Lisboa, Portugal}
\email{p.goncalves@tecnico.ulisboa.pt}

\author{Adriana Neumann}
\address{UFRGS, Instituto de Matem\'atica, Campus do Vale, Av. Bento Gon\c calves, 9500. CEP 91509-900, Porto Alegre, Brasil}
\email{aneumann@mat.ufrgs.br}
\thanks{}

\allowdisplaybreaks

\begin{document}

 \begin{abstract}
We prove a large deviations principle for  the empirical measure of the one dimensional symmetric simple exclusion process in contact with reservoirs. The dynamics of the reservoirs is slowed down with respect to the dynamics of the system, that is, the rate at which the system exchanges particles with the boundary reservoirs is of order $n^{-\theta}$, where $n$ is number of sites in the system, $\theta$ is a non negative parameter, and the system is taken in the diffusive time scaling. Two regimes are studied here, the subcritical $\theta\in(0,1)$ whose hydrodynamic equation is  the heat equation with Dirichlet boundary conditions  and the supercritical    $\theta\in(1,+\infty)$ whose hydrodynamic equation is the heat equation with Neumann boundary conditions. In the subcritical case $\theta\in(0,1)$, the rate function that we obtain matches the rate function corresponding to the case $\theta=0$ which was derived on previous works (see \cite{blm,flm}), but the challenges we faced here are much trickier. In the supercritical case $\theta\in(1,+\infty)$, the rate function is equal  to infinity outside the set of  trajectories which preserve the total mass, meaning that, despite the discrete system exchanges particles with the reservoirs, this phenomena has  super-exponentially small probability in the diffusive scaling limit.
\end{abstract}

\maketitle

\tableofcontents 

\section{Introduction}\label{s1}
Due to its special features and simplicty, the exclusion process became a prototype interacting particle system in Probability and Statistical Mechanics: in one hand it presents an interaction among particles (the hard-core interaction) describing many physical phenomena of interest. On the other hand, it is also a mathematical treatable model, allowing rigorous proofs of those  phenomena. See \cite{kl} on the subject, for instance.

In plain words, the exclusion process is described by independent random walks on some graph under the constraint that at most one particle is allowed to occupy each vertex of the graph. 
Variations of the exclusion dynamics then lead to many different physical situations. One of the most common and relevant is to put the exclusion process in contact with reservoirs, and this has been  widely studied in the literature. In particular, the symmetric exclusion in contact with reservoirs  is the subject of study in this paper.

Recently, in \cite{bmns} it was derived the hydrodynamic limit of the one-dimensional symmetric exclusion process on the box with $n$ sites and in contact with slow reservoirs. That is, the dynamics is given by a superposition of a Kawasaki dynamics and a Glauber dynamics at the end points of the box. More precisely, the symmetric simple exclusion dynamics acts on the bulk, that is the set of points $\{1,\dots, n-1\}$ and at the sites $1$ and $n-1$, particles can be injected/removed to/from the bulk at a rate which is slowed down with respect to the bulk dynamics. More precisely, particles enter (respectively, leave) the system through the left boundary at rate $\alpha /n^{\theta}$ (respectively, $(1-\alpha)/n^{\theta}$) and particles enter (respectively, leave) the system through the right boundary at rate $\beta/ n^{\theta}$ (respectively, $(1-\beta)/ n^{\theta}$). Here, $0<\alpha, \beta<1$ and $\theta \geq 0$ are fixed parameters.

Given that the hydrodynamic limit has been established, which is, in some sense, a law of large numbers for the density of particles, since its limit is deterministic, it is quite natural to ask about its  large deviations. That is, the asymptotic probability to observe rare events (which, \textit{grosso modo}, goes exponentially fast to zero for events that do not contain the expected limit from the law of large numbers). This is precisely what we do here: in this paper we study the large deviations of the model studied in \cite{bmns}, in both the subcritical case $\theta\in(0,1)$ and in the supercritical case $\theta\in(1,+\infty)$. Due to its special characteristics and also for a matter of size of this article, the critical case $\theta=1$ is left to future work.

We describe next the main features of this work, starting with some words  about the super-\break exponential replacement lemmas which are of fundamental importance in the derivation of our results. For $\theta\in(0,1)$ since we are in the regime of Dirichlet boundary conditions,  we need to replace the value of the empirical  measure at the boundary by the value $\alpha$, at the left boundary  and $\beta$, at the right boundary. This can be achieve by taking as reference measure a product measure associated to a continuous profile $g_{\alpha,\beta}$ which is locally equal to $\alpha$ at the left boundary and locally equal to $\beta$ at the right boundary. For $\theta\in(1,+\infty)$,  we need to assure that the profiles which do not preserve the total mass of the system have probability super-exponentially small. These facts help solving the elliptic equation (associated to the weakly asymmetric system), and this provides the correct perturbation to take  in order to observe a chosen profile.

For $\theta\in(0,1)$, the large deviations rate function we obtained  coincides with the large deviations rate function of many previous works, as \cite{BSGLandim}, or \cite{flm} in dimension one and parameter $a=0$, or in \cite{Bodineau_Lagouge} if we do not consider the reaction dynamics as they do. We stress that despite having the same large deviations rate function,  the case $\theta\in(0,1)$  is not a particular case of those aforementioned works, since many of the estimates that we need are harder to obtain. Nevertheless, their exchange rates at the boundary corresponds to taking $\theta=0$ in our rates. At the end, we prove that slowing the exchange rate of  the boundary by $n^{-\theta}$, with $\theta\in(0,1)$ the large deviations behave as in the case $\theta=0$.

For $\theta\in(1,+\infty)$, contrarily to the case $\theta\in(0,1)$,  the large deviations rate function depends on the value of the density profile at the boundary.  The rate function for the case $\theta\in(0,1)$ and $\theta\in(1,+\infty)$ are then written in the following succinct form, as the supremum over the set of possible perturbations $H$ of the price function $J_H(\rho)$, but restricting the set of reachable profiles $\rho$ to distinct sets in each case. In other words, the rate function for $\theta\in(0,1)$ and $\theta\in(1,+\infty)$ are quite similar, in their form, but they have, as natural, different attainable  profiles. The constraint $\rho_t(0)=\alpha$ and $\rho_t(1)=\beta$ defines the set of reachable profiles for $\theta\in (0,1)$, which corresponds to the Dirichlet case,  while  a time-invariant mass  constraint defines the reachable set of profiles for $\theta\in(1,+\infty)$, which corresponds to the Neumann case.  Both sets of reachable profiles are natural if we take into consideration that  the corresponding  hydrodynamic equations have Dirichlet and Neumann boundary conditions, respectively.

In neither the cases $\theta\in(0,1)$ and $\theta\in(1,+\infty)$ the current through the boundary plays any role. This can be explained as follows. For $\theta\in(0,1)$, in the same spirit of \cite{BSGLandim,Bodineau_Lagouge,flm}, a super-exponential replacement lemma at the boundary holds, meaning that the exchange of particles is fast enough to not allow any large deviations in the diffusive scaling. 
On the other hand, for  $\theta\in(1,+\infty)$, the exchange of particles is so slow that any large deviations of the current through the boundary have no strength  to interfere in the large deviations of the density. That is, the current through the boundary disappears super-exponentially fast in the diffusive time scaling.
 
 The paper is structured as follows: In Section~\ref{s2} we give definitions and we state our main results. Section~\ref{s3} contains the necessary super-exponential replacement lemmas which are crucial along the arguments. In Section~\ref{s4} we study the hydrodynamic limit of the associated weakly asymmetric process. In Section ~\ref{s5} it is presented the large deviations upper bound and in Section~\ref{sec6} the lower bound.



\section{Statement of results}\label{s2}

\subsection{The model}

Given $n\geq{1}$, denote  $\Sigma_n=\{1,\ldots,n-1\}$ and consider the state space $\Omega_n:=\{0,1\}^{\Sigma_n}$. Configurations on this state space $\Omega_n$ will be denoted by $\eta$ so that, for $x\in\Sigma_n$,  $\eta(x)=0$ means that the site $x$ is vacant while $\eta(x)=1$ means that the site $x$ is occupied.   We define the infinitesimal generator $\mc L_{n}=\mc L_{n,0}+n^{-\theta}\mc L_{n,b}$  as follows. For any function $f:\Omega_n\rightarrow \bb{R}$,  
\begin{equation}\label{ln0}
\begin{split}
(\mc L_{n,0}f)(\eta)\;=\;
\sum_{x=1}^{n-2}\Big(f(\eta^{x,x+1})-f(\eta)\Big)\,, 
\end{split}
\end{equation}
\begin{equation}\label{lnb}
(\mc L_{n,b}f)(\eta)\;=\;
\sum_{x\in\{1,n-1\}} \Big[{r_x}(1-\eta(x))+(1-r_x)\eta(x)\Big]\Big(f(\sigma^{x} \eta)-f(\eta)\Big)\,,
\end{equation}
with $r_1=\alpha$ and  $r_{n-1}=\beta$. Above,  for $x\in\{1,\ldots, n-2\}$, the configuration $\eta^{x,x+1}$ is obtained from $\eta$ by exchanging the occupation variables $\eta(x)$ and $\eta(x+1)$, i.e.,
\begin{equation}\label{etax,x+1}
(\eta^{x,x+1})(y)\;=\;\left\{\begin{array}{cl}
\eta(x+1)\,,& \mbox{if}\,\,\, y=x\,,\\
\eta(x)\,,& \mbox{if} \,\,\,y=x+1\,,\\
\eta(y)\,,& \mbox{otherwise,}
\end{array}
\right.
\end{equation}
 and for $x\in\{1,n-1\}$ the configuration  $\sigma^x\eta$ is obtained from $\eta$ by flipping  the occupation  variable $\eta(x)$, i.e,
  \begin{equation}\label{sigmax}
(\eta^x)(y)\;=\;\left\{\begin{array}{cl}
1-\eta(y)\,,& \mbox{if}\,\,\, y=x\,,\\
\eta(y)\,,& \mbox{otherwise.}
\end{array}
\right.
\end{equation}
The dynamics of this model can be described in words in the following way. In the bulk, particles move according to continuous time symmetric random walks under the  exclusion rule: whenever a particle tries to jump to an occupied site, such jump is suppressed. Additionally, at the left boundary, particles can be created (resp. removed) at rate $\alpha/ n^\theta$ (resp. at rate $(1-\alpha)/n^\theta$) and at the right boundary, particles can be created (resp. removed) at rate $\beta/ n^\theta$ (resp. at rate $(1-\beta)/ n^\theta$), see Figure~\ref{fig1} for an illustration.
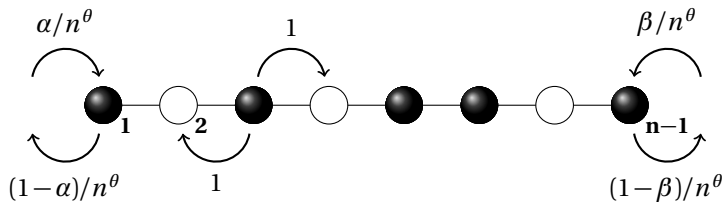
\begin{figure}[H]
\centering
\begin{tikzpicture}
\centerarc[thick,<-](0.5,0.3)(10:170:0.45);
\centerarc[thick,->](0.5,-0.3)(-10:-170:0.45);
\centerarc[thick,->](2.5,-0.3)(-10:-170:0.45);
\centerarc[thick,<-](3.5,0.3)(10:170:0.45);
\centerarc[thick,<-](8.5,-0.3)(-10:-170:0.45);
\centerarc[thick,->](8.5,0.3)(10:170:0.45);
\draw (1,0) -- (8,0);

\shade[ball color=black](1,0) circle (0.25);
\shade[ball color=black](3,0) circle (0.25);
\shade[ball color=black](5,0) circle (0.25);
\shade[ball color=black](6,0) circle (0.25);
\shade[ball color=black](8,0) circle (0.25);

\filldraw[fill=white, draw=black]
(2,0) circle (.25)
(4,0) circle (.25)
(7,0) circle (.25);

\draw (1.3,-0.05) node[anchor=north] {\small  $\bf 1$}
(2.3,-0.05) node[anchor=north] {\small $\bf 2 $}
(8.5,-0.05) node[anchor=north] {\small $\bf n\!-\!1$};
\draw (0.5,0.8) node[anchor=south]{$\alpha/n^\theta$};
\draw (0.5,-0.8) node[anchor=north]{$(1-\alpha)/n^\theta$};
\draw (3.5,0.8) node[anchor=south]{$1$};
\draw (8.5,-0.8) node[anchor=north]{$(1-\beta)/n^\theta$};
\draw (8.5,0.8) node[anchor=south]{$\beta/n^\theta$};
\draw (2.5,-0.8) node[anchor=north]{$1$};
\end{tikzpicture}
\caption{Illustration of  jump rates. The leftmost and rightmost rates are the entrance/exiting rates.}\label{fig1}
\end{figure}
When $\alpha=\beta=\rho$, for which there is no external current induced by the reservoirs, the Bernoulli product measures given by $\nu_\rho\{\eta:\eta(x)=1\}=\rho$ are invariant. However, when $\alpha \neq \beta$, this is no longer true.  Nevertheless, for $\alpha\neq \beta$, there is a unique stationary measure of the system, that we denote by  $\mu_{\text{ss}}$, which is not a product measure. For further properties on this measure we refer the reader to \cite{d}, for instance. In \cite[Theorem 2.2]{bmns}, it is shown that this measure is associated to a profile $\bar{\rho}(\cdot)$ which is stationary with respect to the corresponding hydrodynamic equation.

Fix, once and for all, a time horizon $T>0$. 
We denote by $\{\eta_t: t\in[0,T]\}$  the Markov process with generator $n^2\mc L_n$, omitting the dependence on $n$ to shorten notation. This family of Markov processes  $\{\eta_t\,:\, t\in[0,T]\}$ indexed on $n\in \bb N$ is what we will call the \textit{Exclusion Process with Slow Boundary} (EPSB).

\subsection{Empirical measure}

The so-called  \textit{empirical measure}, which represents the spatial density of particles in the system, is defined by
\begin{equation}\label{eq:emp_mea}
\pi^n(du)\;=\;\pi^n(\eta, du)\;:=\;\frac {1}{n} \sum_{x = 1}^{n\!-\!1}  \eta(x)\, \delta_{\frac{x}{n}}(du)\,,
\end{equation}
where $\delta_{\frac{x}{n}}$ is the Dirac-measure on $x/n\in [0,1]$ and $\eta\in\Omega_n$.  
Note that the empirical measure is a random positive measure on $[0,1]$ with mass bounded by one.
Let
\begin{equation}\label{defM}
\mc M\;=\;\{\,\mu \mbox{ is a positive measure on } [0,1]: \mu([0,1])\leq 1\}\,,
\end{equation}
hence $\pi^n\in\mc M$.
The integral of a function $f:[0,1]\to \bb R$ with respect to the empirical measure is $
\int_0^1 f(u)\,\pi^n(du)=\frac {1}{n} \sum_{x = 1}^{n\!-\!1}  \eta(x)\, f(\pfrac{x}{n})$, for which we will write $\<\pi^n,f\>$. 
The time evolution of the density of particles can be represented by the time evolution of the empirical measure as
\begin{equation*}
\pi^n_t(du)\;=\;\pi^n(\eta_t, du)\;:=\;\frac {1}{n} \sum_{x = 1}^{n\!-\!1}  \eta_t(x)\, \delta_{\frac{x}{n}}(du)\,,
\end{equation*}
where  $\{\eta_t\,:\, t\in[0,T]\}$ is the EPSB. This is the object we are concerned with in this work. 
\subsection{Notations} 
In what follows we present notations to be used everywhere in this paper and we also recall some classical spaces from Analysis.\smallskip

$\bullet$ We will write $\<\cdot,\cdot\>$  to denote both an integral of a function $f$ with respect to a measure $\mu$, that is, 
$\<\mu,f\>=\int_0^1 f(u)\,d\mu(u)$, 
and to denote the inner product on $L^2[0,1]$ given by $\<f,g\>=\int_0^1 f(u)\,g(u)\,du$. The double bracket $\<\!\<\cdot,\cdot\>\!\>$ denotes the inner product in $L^2([0,T]\times[0,1])$ and the corresponding norm is denoted by $\|\cdot\|_{L^2(0,T;(0,1))}$ \smallskip

$\bullet$  Let  $\Ddiscreto=\mc D([0,T],\Omega_n)$ be  the space of trajectories that are right continuous, with left limits and  taking values in $\Omega_n$, which was defined in the beginning of the Section \ref{s2}.
Denote by $\bb P_{\mu_{n}} $ the probability on  $\Ddiscreto$ induced by $\{\eta_t\,:\, t\in[0,T]\}$ and by the initial measure $\mu_n$, and let $\bb E_{\mu_n}$ be the expectation with respect to $\bb P_{\mu_n}$.

Denote by $\DM=\mc D([0,T],\mc M)$   the space of trajectories that are right continuous, with left limits and  taking values in $\mc M$, which was defined in \eqref{defM}. 
Denote by $\bb Q_{\mu_{n}} $ the probability on  $\DM$ induced by
 $\{\pi^n_t\,:\, t\in[0,T]\}$ and by the initial measure $\mu_n$ on $\Omega_n$.
 
Denote by $\DMO$ the subset of $\DM$ consisting in trajectories taking values on measures which have density with respect to the Lebesgue measure between zero and one.\smallskip

$\bullet$ We will denote $C^{i,j} :=C^{i,j}([0,T]\times [0,1])$, the set of functions which are $C^i$ in time and $C^j$ in space. When only one \textit{superindex} appears, it  means that we are considering a function depending only on the space variable. For example, $C^0$ denotes the set of continuous functions $H:[0,1]\to \bb R$.  When a \textit{subindex} $0$ appears, it will restrict the considered set to functions which vanish at the boundary of $[0,1]$. When a subindex $c$ appears, it will restrict the considered set to functions of support compact in $(0,1)$.
For example, by $C^{1,2}_c$ we  mean the subset of $C^{1,2}$ with functions of compact support in $[0,T]\times (0,1)$ and by $C^{1,2}_0$ we  mean the subset of $C^{1,2}$ composed by functions $H$ such that $H(t,0)=H(t,1)=0$ for all $t\geq 0$. 

Define then
\begin{equation}\label{C_theta_set}
\C \;=\; 
\begin{cases}
C_0^{1,2}, \;\textrm{if }\theta\in(0,1),\\
C^{1,2}, \; \textrm{if }\theta\in(1,+\infty).
\end{cases}
\end{equation}

$\bullet$  Given a function $g:[0,T]\times [0,1]$, we sometimes use  $g_t(u)$ to denote  $g(t,u)$. 
It should   not be confounded  with the notation $\partial_t g(t,u)$ for  the time derivative.\smallskip

$\bullet$ The notation $g(n)=O(f(n))$ means  $g(n)$ is bounded from above by $Cf(n)$, where the  constant $C>0$ does not depend on $n$. A presence of \textit{subindexes} in the \textit{big Oh}  means that the constant may depend on those subindexes. Equivalently, $f\lesssim g$ will stand for $f=O(g)$.
The notation $g(n)=o(f(n))$ will stand for 
$\displaystyle \lim_{n\to\infty} g(n)/f(n)=0$.\smallskip

$\bullet$ The indicator function of a set $A$ will be written as $\textbf 1_{A}(u)$, which is
one if $u\in A$ and zero otherwise.\smallskip

$\bullet$ The discrete derivatives and the discrete Laplacian are defined by
\begin{equation}\label{nabla}\nabla_n^+ H_n(\pfrac{x}{n})\;=\;n \Big[H(\pfrac{x+1}{n})-H(\pfrac{x}{n})\Big]\,, \qquad \nabla_n^- H_n(\pfrac{x}{n})\;=\;n \Big[H(\pfrac{x}{n})-H(\pfrac{x-1}{n})\Big]
\,,
\end{equation}
\begin{equation}\label{lapla}\Delta_n H_n(\pfrac{x}{n})\;=\;n^2 \Big[H(\pfrac{x+1}{n})+H(\pfrac{x-1}{n})-2H(\pfrac{x}{n})\Big]\,.
\end{equation}

\begin{definition}[Sobolev Space]\label{Sobolev}\quad
Let  $\mc H^1$ be the set of all locally summable functions $\zeta: (0,1)\to\bb R$ such that
there exists a function $\p_u\zeta\in L^2$ satisfying
$ \<\partial_uG,\zeta\>\,=\,-\<G,\partial_u\zeta\>$,
for all $G\in C^{\infty}_{c}$.
For $\zeta\in\mc H^1$, we define the norm
\begin{equation*}
 \Vert \zeta\Vert_{\mc H^1}\,:=\, \Big(\Vert \zeta\Vert_{L^2}^2+\Vert\partial_u\zeta\Vert_{L^2}^2\Big)^{1/2}\,.
\end{equation*}
Let $\Sob$ be the space of
 all measurable functions
$\xi:[0,T]\to \mc H^1$ such that
\begin{equation*}
\Vert\xi \Vert_{\Sob}^2 \,
:=\,\int_0^T \Vert \xi_t\Vert_{\mc H^1}^2\,dt\,<\,\infty\,.
\end{equation*}
\end{definition}

\begin{remark}\rm  An  equivalent definition  for the Sobolev space 
$\Sob$ is the  set  of bounded functions
$\xi:[0,T]\times \bb T\to \bb R$ such that there exists a function 
$\partial\xi\in L^2([0,T]\times \bb T)$ 
satisfying 
\begin{equation*}
\<\!\< \partial_u H,\xi \>\!\> \;=\;- \<\!\<  H,\partial\xi \>\!\>\,,
\end{equation*}
for all functions $H\in C^{0,1}_c$.
\end{remark}

\subsection{Hydrodynamic limit}

Fix a measurable profile $ \rho_0: [0,1] \rightarrow [0,1 ]$. For each  $n \in \bb N$, let $\mu_n$ be a probability measure on $\Omega_n$.  We say that the sequence $\{\mu_n\}_{n\in \bb N}$ is \textit{associated} with the profile $\rho_0(\cdot)$ if,  for any $ \delta >0 $ and any  $ f \in C^0$, the following limit holds:
\begin{equation}\label{eq3}
\lim_{n\to\infty}
\mu_{n} \Big[\, \eta:\, \Big|\<\pi^n_0,f\>
- \<\rho_0,f\>\Big| > \delta\,\Big] \;=\; 0\,.
\end{equation}

From \cite{bmns} we have the following result:
\begin{theorem}[Hydrodynamic limit for the EPSB, c.f.\ \cite{bmns}]
\label{thm:hid_lim}\quad 
 Suppose that the 
sequence\break $\{\mu_n\}_{n\in \bb N}$ is \textit{associated} with a profile $\rho_0(\cdot)$ in the sense of \eqref{eq3}. 
Then,  for each $ t \in [0,T] $, for any $ \delta >0 $ and any continuous function $ f:[0,1]\to\bb R $, 
\begin{equation*}
\lim_{ n \rightarrow +\infty }
\bb P_{\mu_{n}} \Big[\,\eta_{\cdot} : \Big\vert\<\pi^n_t,\,f\> - \<\rho_t,\,f\>\, \Big\vert
> \delta \,\Big] \;=\; 0\,,
\end{equation*}
 where  $\rho(t,\cdot)$ is:\medskip

$\bullet$ If $0<\theta<1$,  the unique weak solution of the 
 heat equation with    Dirichlet boundary conditions
\begin{equation}\label{hydroeq}
\begin{cases}
\p_t \rho(t,u)= \p_u^2 \rho(t,u)\,, & \textrm{ for } t>0\,,\, u\in (0,1)\,,\\
 \rho(t,0) =\alpha\,, \, \rho(t,1) =\beta & \textrm{ for } t>0\,,\\
 \rho(0,u)=\rho_0(u)\,,& \textrm{ for } u\in [0,1]\,.
\end{cases}
\end{equation}

$\bullet$ If $\theta>1$,  the unique weak solution of the 
  heat equation with    Neumann boundary conditions
\begin{equation}\label{hydroeq_Neumann}
\begin{cases}
\p_t \rho(t,u)= \p_u^2 \rho(t,u)\,, & \textrm{ for } t>0\,,\, u\in (0,1)\,,\\
\p_u \rho(t,0) =\p_u \rho(t,1)=0\,, & \textrm{ for } t>0\,,\\
\rho(0,u)=\rho_0(u)\,,& u\in [0,1]\,.
\end{cases}
\end{equation}

\end{theorem}
In \cite{bmns} the authors  prove that the 
sequence of probability measures $ \{ \bb Q_{\mu_n}\}_{n \in \bb N} $ converges weakly to $ \bb Q $ as $n \rightarrow +\infty $, where $ \bb Q $ is  the probability measure on $\DM $ which gives mass $1$ to the path $ \pi(t,du) = \rho_t(u)du $, being $ \rho_t(\cdot) $ the unique  weak solution of \eqref{hydroeq}. Observe that  Theorem \ref{thm:hid_lim} is a corollary of this result.

\subsection{Large Deviations Principle}\label{sub2.5}
We start by recalling the notion of \textit{energy} in the way as \cite{BSGLandim,flm,FN2017} and many other related papers. 
\begin{definition}\label{energy} For $H\in C^{0,1}_c$, define
	$\mc E_H: \DM\to \bb R\cup\{+\infty\}$ by
	\begin{equation*}
	\mc E_H(\pi)\,=\,\left\{\begin{array}{cl}
	\<\!\< \p_u H,\rho \>\!\>
	- 2 \| H\|_{L^2(0,T;(0,1))} \,, &  \mbox{if}\,\,\,\,\pi\in \DMO\mbox{ and }\pi_t(du)=\rho_t(u)\,du\,,\\ 
	\infty\,, &\mbox{otherwise\,.}
	\end{array}
	\right. 
	\end{equation*}
The energy functional $\mc E:\DM\to \bb R_+\cup\{\infty\}$ is then defined as
	$	\mc E(\pi)=\sup_{H\in C^{0,1}_c}\mc E_H(\pi)$.
\end{definition} 
By the Riesz Representation Theorem, it is  well-known that $\mc E(\pi)<\infty$ implies  $\pi_t=\rho_t(u)du$ with $\rho$ belonging to the Sobolev space $L^2(0,T;\mc H^1(0,1))$, see \cite{FN2017}, for instance.

 Given a profile $\rho\in L^2(0,T;\mc H^1(0,1))$, we define the linear functional $\ell_{H}^\theta(\rho)$ acting on $H\in C^{1,2}$ as
\begin{equation*}
\ell_H^\theta(\rho)     = \<\rho_T,H_T\>-\<\rho_0,H_0\>-\int_0^T \<\rho_s, (\p_s+\Delta)H_s\>\,ds
+ \int_0^T\Big(\beta\p_uH_s(1)- \alpha\p_uH_s(0) \Big)ds
\end{equation*}
if $\theta\in(0,1)$
and
\begin{align*}
\ell_H^\theta(\rho)  &    = \<\rho_T,H_T\>-\<\rho_0,H_0\>-\int_0^T \<\rho_s, (\p_s+\Delta)H_s\>\,ds
+ \int_0^T\Big( \rho_s(1)\p_uH_s(1)- \rho_s(0)\p_uH_s(0) \Big)ds\\
& = \<\rho_T,H_T\>-\<\rho_0,H_0\>-\int_0^T \<\rho_s, \p_sH_s\>\,ds + \int_0^T \<\p_u\rho_s, \p_u H_s\>\,ds
\end{align*}
if $\theta\in(1,+\infty)$.
Let   $\Phi_H(\rho)$ be  the non-negative convex functional acting on $H\in C^{1,2}$ as
\begin{equation}\label{12a}
\begin{split}
& \Phi_H(\rho) \;=\; \int_0^T \< \chi(\rho_s), (\p_u H_s)^2 \>\,ds
\end{split}
\end{equation}
where   $\chi(u)=u(1-u)$ is the so-called \textit{static compressibility} of the system.
Given $H\in C^{1,2}$, we define the functional $J_H^\theta: \DM\to \bb R\cup \{+\infty\}$ by
\begin{equation}\label{functional:J}
J_H^\theta(\pi)\;=\;\begin{cases}
\ell_H^\theta(\rho)-\Phi_H(\rho)\,,&\mbox{if }\pi\in\mathcal F_\theta \text{ and }   \mc E(\pi)<\infty  \text{ with }  \pi_t=\rho_t(u)du\,,\\
+\infty\,,& \mbox{otherwise},
\end{cases}
\end{equation}
where

\begin{equation}\label{Ftheta}
 \mc F_\theta\;:=\;
\begin{cases}
\DM, & \text{ if }  \theta\in (0,1),\\
\big\{\pi\in \DM: \<\pi_t,1\>=\<\pi_0,1\> \,,\, \forall t\in[0,T]\big\}, & \text{ if } \theta\in(1,+\infty).
\end{cases} 
 \end{equation}
 We point out that, for $\theta\in(1,+\infty)$, $\mc F_\theta$ is the set of trajectories   such that the total  mass  is constant in time and also that the   boundary integrals in $\ell^\theta_H(\rho)$ are well-defined  due to the assumption $\rho\in L^2(0,T;\mc H^1(0,1))$ and the notion of \textit{trace} of a Sobolev space, see  for instance \cite{E}. \medskip

We  study in this paper the large deviations of the empirical measure starting the system  from a deterministic configuration $\eta^n$,  such that the sequence $\{\eta^n\}_{n\in \bb N}$ of deltas of Dirac  is  associated to $\gamma(u) du$, where   $\gamma:[0,1]\to[0,1]$ is a continuous  profile bounded away from $0$ and $1$.   The probability and expectation of the process starting from a delta of Dirac measure at $\eta^n$ will be denoted  by $\bb P_{\delta_{\eta^n}}$ and $\bb E_{\delta_{\eta^n}}$, respectively. We define next the large deviations rate function.
\begin{definition}\label{I def}Recall from \eqref{C_theta_set} the definition of $\C$.   Let  $\I(\cdot\,|\gamma):\DM\to\bb R_+\cup\{+\infty\}$ be defined by
 \begin{equation}\label{rate_function}
\I(\pi|\gamma)\;=\; \sup_{H\in \C}\, J_H^\theta(\pi)\,.
 \end{equation}
\end{definition}
The rate functional above is lower semi-continuous with compact level sets in both cases $\theta\in[0,1)$ and $\theta\in(1,+\infty)$. The proof of this fact can be readily adapted from \cite[Theorem 4.7]{LandimTsunoda} taking into account that the set of trajectories with constant mass is a closed set in $\DM$.

 We are now in position to state the main result of this paper. 
Let $\bb Q_{\delta_{\eta^n}}$ be the probability measure induced by the empirical measure when we start the system  from $\eta^n$, where $\{\eta^n\}_{n\in \bb N}$ is a sequence of  deterministic configurations  associated to the continuous  profile $\gamma:[0,1]\to[0,1]$, which is bounded away from $0$ and $1$.
\begin{theorem}\label{LDP}
The sequence of probability  measures $\{\bb Q_{\delta_{\eta^n}}\}_{n\geq 1}$ satisfies the following large deviations principle:
\begin{enumerate}[a)]
  \item  {\bf (Upper bound)} For any   closed subset $\mc C$ of $\DM$, 
 \begin{equation*}
 \varlimsup_{n\to\infty}\pfrac{1}{n}\log \bb Q_{\delta_{\eta^n}}\big[\,\mc C\,\big]
\;\leq\; -\inf_{\pi\in\mc C} \I(\pi|\gamma)\,.
 \end{equation*}
 \medskip
 \item  {\bf (Lower bound)} For any  open subset $\mc O$ of $\DM$,
\begin{equation*}
 \varliminf_{n\to\infty}
\pfrac{1}{n}\log \bb Q_{\delta_{\eta^n}}\big[\,\mc O\,\big]\; \geq \; -\inf_{\pi\in \mc O}\I(\pi|\gamma)\,.
\end{equation*}
\end{enumerate}
\end{theorem}

\section{Superexponential Replacement Lemmas}\label{s3}
We start this section  by stating some important estimates on entropy bounds and Dirichlet forms. For technical reasons, it will be important to fix a  particular case where the profile is locally constant equal to $\alpha$ near zero, locally constant equal to $\beta$ near one, and linearly interpolated elsewhere. We denote once and for all this profile by $\g:[0,1]\to[0,1]$ that is  illustrated in Figure~\ref{fig3}.
\begin{figure}[!htb]
\centering
\begin{tikzpicture}[scale=1,smooth];
\draw[thick,->] (0,-0.5)--(0,4.5) node[left]{$\g$};
\draw[thick,->] (-0.5,0)--(4.5,0);
\draw (0,0) node[below left]{$0$};
\draw (2pt,4)--(-2pt,4) node[left]{$1$};

\draw[ultra thick] (0,1)--(0.5,1)--(3.5,3)--(4,3);

\draw (0,1) node[left]{$\alpha$};

\draw[densely dashed] (0.5,1) -- (0.5,0) node[below]{$\delta$}; 
\draw[densely dashed] (3.5,3) -- (3.5,0) node[below]{$1\!-\!\delta$}; 
\draw[densely dashed] (4,3) -- (4,0) node[below]{ $1$}; 
\draw[densely dashed] (3.5,3) -- (0,3) node[left]{$\beta$}; 
\end{tikzpicture}
\caption{Profile $\g$. Note that it depends on $\delta$, which is fixed and whose specific value does not play any role.}\label{fig3}
\end{figure}
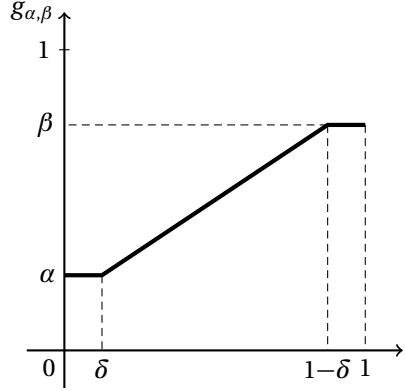
 
By $\nu_{\g(\cdot)}^n$ we denote the \textit{slow varying Bernoulli product measure}  on $\Omega_n$ with parameters given by the profile $\g$, that  is,
\begin{equation}\label{eq:initial_measure}
\nu_{\g(\cdot)}^n\big\{\eta\in \Omega_n\,:\, \eta(x) =1 \text{ for all }x\in D \big\}\;=\;\prod_{x\in D}\g\big(\pfrac{x}{n}\big)\,,\quad \forall\, D\subset \Sigma_n\,.
\end{equation}

\subsection{Entropy bounds and estimates on Dirichlet forms}
For a density function $f:\Omega_n \to [0,\infty)$ with respect to $\nu_{\g(\cdot)}^{n} $ we define
\begin{equation*}
D_{n}(\sqrt{f},\nu_{\g(\cdot)}^{n} )\;:=\; D_{n,0}(\sqrt{f},\nu_{\g(\cdot)}^{n} )+D_{n,b}(\sqrt{f},\nu_{\g(\cdot)}^{n} )\,,
\end{equation*}
where 
\begin{align}\label{eq:bulk_dir_form}
D_{n,0}(\sqrt{f},\nu_{\g(\cdot)}^{n} )& \; :=\; \sum_{x\in\Sigma_n}\, \Big\langle  1,\,\left(\sqrt {f(\eta^{x,x+1})}-\sqrt {f(\eta)}\right)^{2} \Big \rangle_{\nu_{\g(\cdot)}^{n}}\,,\\ 
D_{n,b}(\sqrt{f},\nu_{\g(\cdot)}^{n} )& \;:=\;\frac{1}{n^\theta}\!\!\!\!\sum_{x\in\{1,n-1\}} \!\!\!\! \Big\langle r_x(1-\eta(x))+(1-r_x)\eta(x),\left(\!\!\sqrt{f(\eta^{x})}-\sqrt {f(\eta)}\right)^{2} \!\Big\rangle_{\nu_{\g(\cdot)}^{n}}\label{eq:bound_dir_form}
\end{align} 
and $r_x$ has been already defined in \eqref{lnb}.
  Our first goal is to express a relationship between the Dirichlet form defined by $\langle \mc L_{n}\sqrt{f},\sqrt{f} \rangle_{\nu_{\g(\cdot)}^{n}}$ and $
D_{n}(\sqrt{f},\nu_{\g(\cdot)}^{n} )$. We claim that  
\begin{equation}\label{dir_est:ini_prof}
\begin{split}
\langle \mc L_{n}\sqrt{f},\sqrt{f} \rangle_{\nu_{\g(\cdot)}^n} & \;\lesssim\; -D_{n}(\sqrt{f},\nu_{\g(\cdot)}^n) + \sum_{x=1}^{n-1}\Big(\g(\pfrac{x+1}{n})-\g(\pfrac{x}{n})\Big)^2\\
& \;\lesssim\; -D_{n}(\sqrt{f},\nu_{\g(\cdot)}^n) + \frac{1}{n}\\
\end{split}
\end{equation}
The second inequality above is readily deduced from the definiton of $\g$. To prove the first inequality,   we recall the following lemma from \cite{BGJO,G_lec_not}.
\begin{lemma}
\label{lemmaleft0}
Let $T :   \Omega_n \to \Omega_n$ be a map  and let $c: \eta \to c(\eta)$ be a positive local function. Let $f$ be a density with respect to a probability  measure $\mu$ on $\Omega_n$. Then
\begin{align}
&\left\langle c (\eta) [ \sqrt{f(T (\eta))} -\sqrt{f(\eta)}]\; ,\;\sqrt{f (\eta)} \right\rangle_{\mu} \; \lesssim\;   -\int c (\eta)\left(\left[ \sqrt{f(T (\eta))}\right]-\left[ \sqrt{f(\eta)}\right]\right)^{2}  d\mu\notag \\
& + \int \dfrac{1}{c (\eta)}\left[c (\eta)-c (T(\eta)) \dfrac{\mu (T(\eta))}{\mu(\eta)} \right]^{2}\left(\left[ \sqrt{f(T(\eta))}\right]+\left[ \sqrt{f(\eta)}\right]\right)^{2}  d\mu\,.\label{use_comp}
\end{align}
\end{lemma} 
As a consequence of the previous  lemma,  taking $\mu=\nu_{\g(\cdot)}^n$ we have that
\begin{equation*}
\begin{split}
&\left\langle \mc L_{n,0}\sqrt{f},\sqrt{f}\right \rangle_{\nu_{\g(\cdot)}^n}  \;\lesssim\;  -D_{n,0}(\sqrt{f},\nu_{\g(\cdot)}^n) + \sum_{x=1}^{n-1}\Big(\g(\tfrac xn)-\g(\tfrac {x+1}n)\Big)^2\\
& \left\langle \mc L_{n,b}\sqrt{f},\sqrt{f} \right\rangle_{\nu_{\g(\cdot)}^n}\; \lesssim\;  -D_{n,b}(\sqrt{f},\nu_{\g(\cdot)}^n) + \frac{1}{n^{\theta}}\Big\{\Big(\g(\tfrac 1n)-\alpha\Big)^2+\Big(\g(\tfrac {n-1}n)-\beta\Big)^2\Big\}
\end{split}
\end{equation*}
for any density $f$ with respect to $\nu_{\g(\cdot)}^n$. We leave of details of deriving the above inequalities  to the reader. 
We stress that the profile $\g(\cdot)$ is assumed to satisfy the conditions described below \eqref{eq:initial_measure}, hence the error 
  coming from the bulk dynamics is of order $O(\tfrac 1n)$. In the case $\theta\in (1,+\infty)$ we do not need to impose any extra condition on the profile $\g(\cdot)$ in order to have  the bound given by \eqref{dir_est:ini_prof} since the factor $\tfrac {1}{n^\theta}$  is enough to control this term. On the other hand, the case $\theta\in(0,1)$ indeed requires that the profile $\g(\cdot)$ is equal to $\alpha$ (resp. $\beta$) at $0$ (resp. $1$) and locally constant in a neighborhood of the boundary so that we can control the error. 
\subsection{Replacement lemmas and energy estimates}
In this section we prove  the replacement lemmas required  to write down the Radon-Nikodym derivative as a function of the empirical measure as well some energy estimates.  
Before we proceed we introduce the notion of the  empirical average on a box around $x$. 	By abuse of notation, let   $\eps n$ denotes $\lfloor \eps n\rfloor$, the integer part of $\eps n$.
\begin{definition} For any $x\in \Sigma_n$  and $\eps> 0$ that satisfy $x+\eps n\in \Sigma_n$ we denote by $\eta^{\eps n}(x)$  the centred average  on a box of size $\eps n$ situated to the right or to the  left of the site $x \in \Sigma_n$, that is,
	\begin{numcases}{\eta^{\eps n}(x)=}
		\frac{1}{\eps n}\sum_{z=x+1}^{x+\eps n}\eta(z)\,,& if  $ x\in\{1,\dots,n-1-\eps n\}\,,$\nonumber\\
		\frac{1}{\eps n}\sum_{z=x-\eps n}^{x-1}\eta(z)\,,& if $ x\in\{n-1-\eps n, \dots, n-1\}\,.$\label{eq:emp_average}
		\end{numcases}
\end{definition}

\begin{lemma}\label{lemma:3.2}
Let   $\psi=\psi_{x,\eps,n}:\Omega_n\to\bb R$ be a uniformly bounded function on $n$ and $\eps$ which is invariant for the map $\eta\mapsto \eta^{y,y+1}$ for any $y \in\{x+1,\ldots, x+\eps n\}$, that is, $\psi(\eta)=\psi(\eta^{y,y+1})$ for any $y \in\{x+1,\ldots, x+\eps n\}$. Then, for any density $f$ with respect to $\nu_{\g(\cdot)}^n$, for any $n\geq 1$, for any $\varepsilon>0$ and for  any positive constant $A$, it holds that
\begin{equation*}
\left\vert \left\langle\psi(\eta) \big[\eta(x)-\eta^{\eps n}(x)\big],f\right\rangle_{\nu_{\g(\cdot)}^{n}} \right\vert \; \lesssim\; \tfrac{1}{A}  D_n(\sqrt{f},\nu_{\g(\cdot)}^{n})+ A\varepsilon n+\varepsilon\,.
\end{equation*}
\end{lemma}
 \begin{proof}
We present the proof only for the case  $x\in\{1,..., n-1-\eps n\}$ since  the remaining case is  analogous.  Note that 
 \begin{equation*}
\eta(x)-\eta^{\eps n}(x)\;=\;\frac{1}{\eps n}\sum _{y=x+1}^{x+\eps n} \sum_{z=x}^{y-1} \eta(z)-\eta(z+1)\,.
\end{equation*}
By writing the term  $\eta(z)-\eta(z+1)$ as twice its half and performing the change of variables $\eta$ into $\eta^{z,z+1}$, for each $z$, we have that
\begin{equation*}
\begin{split} 
 \left\langle\psi(\eta)\big(\eta(z+1)-\eta(z)\big), f(\eta)\right\rangle_{\nu_{\g(\cdot)}^{n}}  \;=\;& \dfrac{1}{2} \left\langle\psi(\eta)\big(\eta(z+1)-\eta(z)\big),f(\eta)-f( \eta^{z, z+1})\right\rangle_{\nu_{\g(\cdot)}^{n}} \\
+&\dfrac{1}{2} \left\langle\psi(\eta)\big(\eta(z+1)-\eta(z)\big),f(\eta)+f(\eta^{z, z+1})\right\rangle_{\nu_{\g(\cdot)}^{n}}.
\end{split}
\end{equation*}
By using  that for any $a,b \ge 0$, $(a-b)=(\sqrt a -\sqrt b)(\sqrt a +\sqrt b)$, from Young's inequality, we have, for any positive constant $A$, that 
\begin{align}
 \Big\vert \left\langle\psi(\eta) \big(\eta(x)-\eta^{\eps n}(x)\big),f\right\rangle_{\nu_{\g(\cdot)}^{n}}\Big\vert 
& \lesssim \;   \dfrac{A}{ \eps n} \sum _{y=x+1}^{x+\eps n} \sum _{z=x}^{y-1}\left\langle\big(\eta(z+1)-\eta(z)\big)^2,\Big(\sqrt {f(\eta)}+\sqrt {f(\eta^{z, z+1})}\Big)^2\right\rangle_{\nu_{\g(\cdot)}^{n}} \notag\\
&+\dfrac{1 }{A\eps n}\sum _{y=x+1}^{x+\eps n}  \sum _{z=x}^{y-1} \left\langle 1,\Big(\sqrt {f(\eta)}-\sqrt {f( \eta^{z, z+1})}\Big)^2\right\rangle_{\nu_{\g(\cdot)}^{n}}\notag \\
&+\dfrac{1}{\eps n} \sum _{y=x+1}^{x+\eps n} \left| \sum _{z=1}^{y-1}\left\langle \psi(\eta)\big(\eta(z+1)-\eta(z)\big) ,f(\eta)+f(\eta^{z, z+1})\right\rangle_{\nu_{\g(\cdot)}^{n}} \right|.\label{eq:term1}
\end{align}
Note  that the second term on the right hand side of last display  is bounded from above by\break $\frac 1A D_n (\sqrt f,\nu_{\g(\cdot)}^n) $. Since there is at most a particle per site and since $f$ is a density, the first term at the right hand side of last display is bounded from above by $A\epsilon n$. Finally, to estimate  the third term on the right hand side of last display, we note that, since $\g(\cdot)$ is Lipschitz and there is at most a particle per site,  it is not complicated to show that 
\begin{equation*}
\begin{split}
&\sum_{z=1}^{y-1}
\left\vert \left\langle \psi(\eta)\big(\eta(z+1)-\eta(z)\big) ,f(\eta)+f(\eta^{z, z+1})\right\rangle_{\nu_{\g(\cdot)}^{n}}\right\vert\; \lesssim \; \sum_{z=1}^{y-1}\Big|\g\big(\tfrac{z+1}{n}\big)-\g\big(\tfrac{z}{n}\big)\Big|\;\lesssim \;y/n\,,
\end{split}
\end{equation*}
from where the proof ends. 
\end{proof}

In what follows $\varphi^n$ is a sequence of functions in $C^{0,0}$ with  uniformly bounded supremum norm. Define, for all $\theta\geq 0$,
\begin{equation}\label{eq:V_0}
V_{\eps,0}^{\theta,\varphi^n}(\eta_s,s)\;=\;\frac{1}{n}\sum_{x=1}^{n-2}\varphi_s^n({u_x})\bigg\{\frac{\big(\eta_s(x)-\eta_s(x+1)\big)^2}{2}-\chi\big(\eta_s^{\eps n}(x)\big)\bigg\}\,,
\end{equation}
where $\chi$ is the static compressibility of the system defined in Subsection~\ref{sub2.5}.
Although this expression does not depend on $\theta$, we keep $\theta$ in  the notation to make short some statements in the sequel.
For $ x\in\{1,n-1\}$, let
\begin{equation} \label{eq:V_theta}
V_{\eps,x}^{\theta,\varphi^n}(\eta_s,s)\;=\;\begin{cases}
\varphi^n_s(u_x)\,\big[\eta_s(x)-r_x\big]\,,& \mbox{ if } \theta\in(0,1)\,,\\
\varphi^n_s(u_x)\,\big[\eta_s(x)-\eta_s^{\eps n}(x)\big]\,,& \mbox{ if } \theta\in(1,+\infty)\,,\\
\end{cases}
\end{equation}
where $r_1=\alpha$, $r_{n-1}=\beta$ and $\eta^{\eps n} (x)$ was defined in \eqref{eq:emp_average}.

 \begin{proposition}\label{lem:rep_lem_sup_p}
For any  $t\in[0,T]$, any  $\theta\geq 0$  and any    $ x=0,1,n-1$, we have that
\begin{equation*}
\varlimsup _{\varepsilon \downarrow 0}\varlimsup _{n\rightarrow \infty}\frac 1n\log \bb P_{\nu^n_{\g(\cdot)}}\left[\Big|\int_0^tV_{\eps,x}^{\theta,\varphi^n}(\eta_s,s)\, ds\Big|>\delta\right] \;=\;-\infty\,,
\end{equation*}
for all $\delta>0$. 
\end{proposition}
\begin{proof}
Note that, for $a_n\to+\infty$ and $b_n,c_n>0$, 
\begin{equation}\label{sum_log_super}
\varlimsup_{n\to+\infty} \frac{1}{a_n}\log(b_n+c_n)\;=\;\max\Big\{\varlimsup_{n\to+\infty} \frac{1}{a_n}\log b_n , \varlimsup_{n\to+\infty} \frac{1}{a_n}\log c_n\Big\}\,.
\end{equation}
Using this fact, in order to prove \eqref{eq38super} it is enough to show that estimate  without the absolute value.  By the exponential Chebychev's  inequality, this  probability (without the absolute value) is bounded from above by 
\begin{equation*}
\exp\{-C\delta n\} \,\bb E_{\nu^n_{\g(\cdot)}}\left[\exp\Big\{Cn\int_0^tV_{\eps,x}^{\theta,\varphi^n}(\eta_s,s)\, ds\Big\}\right] \,,
\end{equation*}
for   any  $C>0$.  From Feynman-Kac's formula, last expectation is  bounded from above by 

\begin{equation*} 
\exp\bigg\{\int_0^t \sup _{f}\Big\{\langle Cn V_{\eps,x}^{\theta,\varphi^n}(\eta,s),f\rangle_{\nu^{n}_{\g(\cdot)}}   +n^2\langle \mc L_{n}\sqrt{f},\sqrt{f} \rangle_{\nu_{\g(\cdot)}^n} \Big\}\, ds\bigg\}\,,
\end{equation*}
where the supremum is carried over all the densities $f$ with respect to $\nu_{\g(\cdot)}^n$. 
Up to here we have  
\begin{align}
\frac 1n\log \bb P_{\nu^n_{\g(\cdot)}}&\left[\Big|\int_0^tV_{\eps,x}^{\theta,\varphi^n}(\eta_s,s)\, ds\Big|>\delta\right]\notag\\&\leq\; -C\delta + \int_0^t \sup _{f}\Big\{\langle CV_{\eps,x}^{\theta,\varphi^n}(\eta,s),f\rangle_{\nu^{n}_{\g(\cdot)}}   +n\langle \mc L_{n}\sqrt{f},\sqrt{f} \rangle_{\nu_{\g(\cdot)}^n} \Big\}\, ds\,.\label{eq311}
\end{align}

Due to \eqref{dir_est:ini_prof}, the last expression is bounded from above by  a constant times 
\begin{equation} 
-C\delta + \int_0^t \sup _{f}\Big\{\langle C V_{\eps,x}^{\theta,\varphi^n}(\eta,s),f\rangle_{\nu^{n}_{\g(\cdot)}}   -  n D_{n}(\sqrt{f},\nu_{\g(\cdot)}^n)+1 \Big\}\, ds\,.
\end{equation}
The next step is to obtain a relationship between the two first parcels inside  the supremum above, which has been provided in  Lemma~\ref{lemma:3.2}. Last display can be  bounded from above by
\begin{equation} 
-C\delta + t\, \sup _{f}\Big\{\frac{C}{A}  D_n(\sqrt{f},\nu_{\g(\cdot)}^{n})+ CA\varepsilon n+C\varepsilon -  n D_{n}(\sqrt{f},\nu_{\g(\cdot)}^n)+1\Big\}\,.
\end{equation}
Choosing $A=\tfrac Cn $ on the previous expression, we get $ -C\delta  +t (\varepsilon C^2 +\varepsilon C+1)$, so  taking $\varepsilon \to 0$, we get $ -C\delta  +t$. And then $C\to+\infty$  concludes  the proof, because   $t\in [0,T]$ and $\delta>0$ are fixed.
\end{proof}

In possess of the previous results, it is a standard procedure to derive the (superexponential) energy as written below. One can  follow the arguments as in \cite{FN2017}, for instance.    
\begin{proposition}\label{-l_p} For a function  $H\in 
	C^{0,1}_c$ and $\ell\in \bb R$ fixed, the following inequality holds:
	\begin{equation*}
	\varlimsup_{\eps\downarrow 0}\varlimsup_{n\to \infty}\pfrac{1}{n}
	\log \bb P_{\nu_{\g(\cdot)}^n}\Big[\,\mc E_H(\pi^n*\ioe)\geq \ell\,\Big]\;\leq\; -\ell\,.
	\end{equation*} 
\end{proposition}
\begin{corollary} \label{cor_super_energy_p} For $k\in\bb N$, for  functions  $\{H_j\}_{1\leq j\leq k}$ in $ 
	C^1_c$, and   $\ell\in \bb R$ fixed, we have
\begin{equation*}
\varlimsup_{\eps\downarrow 0}\varlimsup_{n\to\infty}\pfrac 1n \log
\bb P_{\nu_{\g(\cdot)}^n}\Big[\max_{1\leq j\leq k}\mc E_{H_j}\big(\pi^n*\ioe \,\big)
\;\geq\; \ell\,\Big]\;\leq\; -\ell\,.
\end{equation*}
\end{corollary}
We can now move towards super-exponential replacement lemmas for the system starting from the configuration $\eta^n$ associated to the profile $\gamma$. Since 
\begin{align*}
 \frac{\dradon\bb P_{\delta_{\eta^n}}}{\dradon\bb P_{\nu_{\g(\cdot)}^n}} \;=\;
 \frac{\dradon\delta_{\eta^n}}{\dradon\nu_{\g(\cdot)}^n}\;=\;\one_{\eta^n}(\eta)\prod_{x=1}^{n-1}\Big(\g\big(\pfrac{x}{n}\big)\Big)^{\eta(x)}\Big(1-\g\big(\pfrac{x}{n}\big)\Big)^{1-\eta(x)}\,,
 \end{align*}
 we deduce that there exists a constant $c_{\alpha,\beta}>0$ such that 
 \begin{align*}
 \bigg\vert\frac{\dradon\bb P_{\delta_{\eta^n}}}{\dradon\bb P_{\nu_{\g(\cdot)}^n}}\bigg\vert \;\leq\; e^{c_{\alpha,\beta}n}\,.
 \end{align*}
From the  inequality above,  we have  $
\bb P_{\delta_{\eta^n}}[\cdot]\leq \exp\{c_{\alpha,\beta} n\} \bb P_{\nu_{\g(\cdot)}^n}[\cdot]$. Then, from Proposition~\ref{lem:rep_lem_sup_p}, Proposition~\ref{-l_p} and Corollary~\ref{cor_super_energy_p} we obtain the analogous  results when the system starts from $\eta^n$, that is:
 \begin{proposition} \label{lem:rep_lem_sup}
For any  $t\in[0,T]$, any  $\theta\geq 0$  and any    $ x=0,1,n-1$, we have that
\begin{equation}\label{eq38super}
\varlimsup _{\varepsilon \downarrow 0}\varlimsup _{n\rightarrow \infty}\frac 1n\log \bb P_{\delta_{\eta^n}}\left[\Big|\int_0^tV_{\eps,x}^{\theta,\varphi^n}(\eta_s,s)\, ds\Big|>\delta\right] \;=\;-\infty\,,
\end{equation}
for all $\delta>0$. 
\end{proposition}
\begin{proposition}\label{-l} For a function  $H\in 
	C^{0,1}_c$ and $\ell\in \bb R$ fixed, the following inequality holds:
	\begin{equation*}
	\varlimsup_{\eps\downarrow 0}\varlimsup_{n\to \infty}\pfrac{1}{n}
	\log \bb P_{\delta_{\eta^n}}\Big[\,\mc E_H(\pi^n*\ioe)\geq \ell\,\Big]\;\leq\; -\ell+c_{\alpha, \beta}\,.
	\end{equation*} 
\end{proposition}
\begin{corollary} \label{cor_super_energy} For $k\in\bb N$, for  functions  $\{H_j\}_{1\leq j\leq k}$ in $ 
	C^1_c$, and   $\ell\in \bb R$ fixed, we have
\begin{equation*}
\varlimsup_{\eps\downarrow 0}\varlimsup_{n\to\infty}\pfrac 1n \log
\bb P_{\delta_{\eta^n}}\Big[\max_{1\leq j\leq k}\mc E_{H_j}\big(\pi^n*\ioe \,\big)
\;\geq\; \ell\,\Big]\;\leq\; -\ell+c_{\alpha, \beta}\,.
\end{equation*}
\end{corollary}

\section{Perturbed Process}\label{s4}

In order to derive a large deviations principle, it is natural to start with a  class of perturbations of the original process, which can lead the system to converge to any profile, or at least to any profile in a  dense set.

\textit{A priori}, it is not clear what is the natural set of perturbations of the system that one has to consider. For this reason, it makes sense to study at first a quite general set of perturbations. We will decide \textit{a posteriori} which is the correct set of perturbations based on the following criterion: the Radon-Nikodym derivative should be (close to) a function of the empirical measure  and the elliptic equation associated to the perturbed process must have a (unique) solution. This will be made clear along the text. 
Of course, we could have started from the correct set of perturbations, but we chose not doing so for  sake of clarity.

Fix two functions $H$ and $G$.  The general perturbed process we hence consider is the weakly asymmetric  exclusion process with slow boundary (WAEPSB), which we define through the generator\break $\mc L^{H,G}_{n,t}=\mc L^{H,t}_{n,0}+n^{-\theta}\mc L^{G,t}_{n,b}$ acting on functions 
$f:\Omega_n\rightarrow \bb{R}$ as:
\begin{align}\label{ln0H} 
(\mc L_{n,0}^{H,t}f)(\eta)& \;=\;
\sum_{x=1}^{n-2}e^{(\eta(x)-\eta(x+1))\big(H_t(\frac{x+1}{n})-H_t(\frac{x}{n})\big)}\Big(f(\eta^{x,x+1})-f(\eta)\Big)\,,\\
(\mc L_{n,b}^{G,t}f)(\eta)& \;=\;
	\sum_{x\in\{1,n-1\}} \!\Big[e^{G_t(\frac xn)}r_x(1-\eta(x))+e^{-G_t(\frac xn)}(1-r_x)\,\eta(x)\Big]\Big(f( \eta^x)-f(\eta)\Big)\,,\label{lnbH} 
\end{align}
 where  $\eta^{x,x+1}$ was defined in \eqref{etax,x+1}, $r_1=\alpha$, $r_{n-1}=\beta$ and $\sigma^x\eta$ was defined in \eqref{sigmax}.
The role of  the function $H$ and $G$ is  to introduce  weak asymmetries at  the bulk  and at the boundary, respectively. We  assume here that $H\in C^{1,2}$ and that $G$ is $C^1$ in time.
\begin{figure}[!htb]
\centering
\begin{tikzpicture}
\centerarc[thick,<-](0.5,0.3)(10:170:0.45);
\centerarc[thick,->](0.5,-0.3)(-10:-170:0.45);
\centerarc[thick,->](4.5,-0.3)(-10:-170:0.45);
\centerarc[thick,<-](4.5,0.3)(10:170:0.45);
\centerarc[thick,->](8.5,-0.3)(-10:-170:0.45);
\centerarc[thick,<-](8.5,0.3)(10:170:0.45);
\draw (1,0) -- (8,0);

\shade[ball color=black](1,0) circle (0.25);
\shade[ball color=black](3,0) circle (0.25);
\shade[ball color=black](5,0) circle (0.25);
\shade[ball color=black](6,0) circle (0.25);
\shade[ball color=black](8,0) circle (0.25);

\filldraw[fill=white, draw=black]
(2,0) circle (.25)
(4,0) circle (.25)
(7,0) circle (.25);

\draw (1.3,-0.05) node[anchor=north] {\small  $\bf 1$}
(2.3,-0.05) node[anchor=north] {\small $\bf 2 $}
(4.35,-0.05) node[anchor=north] {\small $\bf x $}
(8.5,-0.05) node[anchor=north] {\small $\bf n\!-\!1$};
\draw (0.5,0.8) node[anchor=south]{$\displaystyle\frac{\alpha}{n^\theta} e^{G_t\big(\tfrac{1}{n}\big)}$};
\draw (0.5,-0.8) node[anchor=north]{$\displaystyle\frac{(1-\alpha)}{n^\theta} e^{-G_t\big(\tfrac{1}{n}\big)}$};
\draw (4.5,0.8) node[anchor=south]{$\displaystyle e^{\tfrac{1}{n}\nabla^+_n H_t(\frac{x}{n})}$};
\draw (8.5,-0.8) node[anchor=north]{$\displaystyle\frac{\beta}{n^\theta} e^{G_t\big(\tfrac{n-1}{n}\big)}$};
\draw (8.5,0.8) node[anchor=south]{$\displaystyle\frac{(1-\beta)}{n^\theta} e^{-G_t\big(\tfrac{n-1}{n}\big)}$};
\draw (4.5,-0.8) node[anchor=north]{$\displaystyle e^{-\tfrac{1}{n}\nabla^+_n H_t(\frac{x}{n})}$};
\end{tikzpicture}
\caption{Illustration of  jump rates for the perturbed process.}\label{fig4}
\end{figure}
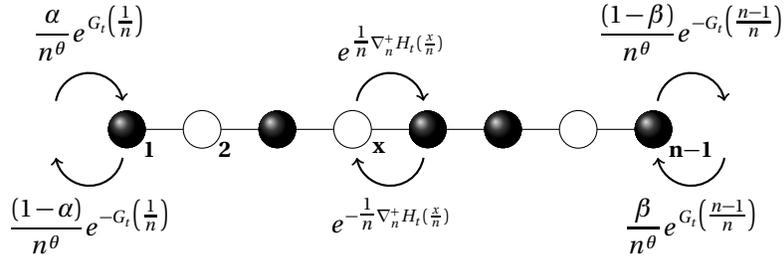

  The general formula for the Radon-Nikodym derivative between two Markov processes $\bb P$ and $\overline{\bb P}$  can be found in \cite{kl},  and  is given by
\begin{equation}\label{der_rN}
\frac{\dradon \bb P}{\dradon \overline{\bb P}}\Bigg|_{\mc F_t}\;=\;
\exp\Bigg\{-\Big(\int_0^t\big[\lambda(X_s)-\overline{\lambda}(X_s)\big]ds-\sum_{s\leq t}\log\frac{\lambda(X_{s^-})p(X_{s^-},X_s))}{\overline{\lambda}(X_{s^-})\overline{p}(X_{s^-},X_s)}\Big)\Bigg\}\,,
\end{equation}
where $\lambda$ and $\overline{\lambda}$  are the waiting times and $p(\cdot,\cdot)$ and $\overline{p}(\cdot,\cdot)$ are the   transition probabilities of $\bb P $ and $\overline{\bb P}$, respectively.  Above $\mc F_t$ stands for the natural filtration. It is  important to remark that the classical book \cite[formula (2.6), page 320]{kl} has 
indeed a typo on this formula:  the  formula for the Radon-Nikodym derivative   $\frac{\dradon \bb P}{\dradon \overline{\bb P}}\big|_{\mc F_t}$ as presented there was, in fact, the formula for  $\frac{\dradon \overline{\bb P}}{\dradon \bb P}\big|_{\mc F_t}$. Since one is the inverse of the other, this is of course, just a typo. However, since such formula is the key stone to obtain the rate function, we think it is worth explaining this typo in details and we leave this discussion  to Appendix~\ref{Appendix}. In what follows  we  compute the Radon-Nikodym derivative 
$\frac{\dradon\bb P_{\subm}}{\dradon\bb P_{\subm}^{H,G}}\Big|_{\mc F_t}$, where:

	$\bullet$ The measure $\bb P_{\subm}$ is induced by the Markov process with infinitesimal generator $\mc L_n=\mc L_{n,0}+n^{-\theta}\mc L_{n,b}$, see  \eqref{ln0} and \eqref{lnb},  starting from the configuration $\eta^n$.
	
	$\bullet$ The measure $\bb P_{\eta^n}^{H,G}$ is induced by the Markov process with infinitesimal generator $\mc L_{n,t}^{H,G}$\break $=\mc L_{n,0}^{H,t}+n^{-\theta}\mc L_{n,b}^{G,t}$, see   \eqref{ln0H} and \eqref{lnbH}, starting from the configuration $\eta^n$.

Having the expression \eqref{der_rN} for the Radon-Nikodym derivative between two processes, we first deal  with the sum
\begin{equation}\label{sum_log}
-\sum_{s\leq t}\log\frac{\lambda(X_{s^-})p(X_{s^-},X_s))}{\overline{\lambda}(X_{s^-})\overline{p}(X_{s^-},X_s)}\,.
\end{equation}
Evaluating   the parameters $\lambda,\bar\lambda, p,\bar p$ for our model,  \eqref{sum_log}  becomes  
\begin{align}
&\sum_{s\leq t}\bigg[G_s(\tfrac{1}{n})\Big\{\textbf{1}_{\{\eta_{s^-}(1)=0, \eta_{s}(1)=1, \eta_{s^-}(2)=\eta_{s}(2)\}}-\textbf{1}_{\{\eta_{s^-}(1)=1, \eta_{s}(1)=0, \eta_{s^-}(2)=\eta_{s}(2)\}}\Big\}\notag\\
&+\!G_s(\tfrac{n-1}{n})\Big\{\textbf{1}_{\{\eta_{s^-}(n-1)=0, \eta_{s}(n-1)=1, \eta_{s^-}(n-2)=\eta_{s}(n-2)\}}
-\textbf{1}_{\{\eta_{s^-}(n-1)=1, \eta_{s}(n-1)=0, \eta_{s^-}(n-2)=\eta_{s}(n-2)\}}\Big\}\notag\\
&+\sum_{x=1}^{n-2}\tfrac{1}{n}\nabla_n^+H_s(\tfrac{x}{n})\Big\{\textbf{1}_{\{\eta_{s^-}(x)=1, \eta_{s}(x)=0, \eta_{s^-}(x+1)=0, \eta_{s}(x+1)=1\}}\notag\\
&-\textbf{1}_{\{\eta_{s^-}(x)=0, \eta_{s}(x)=1, \eta_{s^-}(x+1)=1, \eta_{s}(x+1)=0\}}\Big\}\bigg]\,,\label{eq4.8}
\end{align}
where
$\nabla_n^+$ is  the discrete derivative defined in \eqref{nabla}.
To shorten the expression above, we define now some currents. 

For $x\in\{1,\ldots,n-2\}$, denote by $J^n_{x,x+1}(t)$, the current through the edge $\{x, x + 1\}$, that
is, the total number of particles that have jumped from $x$ to $x + 1$ minus the total
number of particles that have jumped from $x + 1$ to $x$ up to time $t$. The quantity  $J^n_{0,1}(t)$ denotes the current at site $1$, that is, the total number of particles created  at the site $1$ minus the total number of  particles destroyed at the site $1$  up to time $t$, while $J^n_{n-1,n}(t)$ denotes the current at site $n-1$, that is, the total number of  particles destroyed at the site $n-1$ minus the total number of particles created  at the site $n-1$  up to time $t$.
These notions allow to rewrite the expression \eqref{eq4.8} simply as
\begin{equation}\label{eq4.6}
\int_0^t\Big\{G_s(\tfrac{1}{n})\partial_sJ^n_{0,1}(s)-G_s(\tfrac{n-1}{n})\partial_sJ^n_{n-1,n}(s)+\sum_{x=1}^{n-2} \tfrac{1}{n}\nabla_n^+ H_s (\tfrac{x}{n})\partial_sJ^n_{x,x+1}(s)\Big\}\,ds\,.
\end{equation}
From an integration by parts in time, a  summation by parts in space  and  the conservation law $\eta_t(x)-\eta_0(x)=J^n_{x-1,x}(t)-J^n_{x,x+1}(t)$, we infer that \eqref{eq4.6} is the same as 
\begin{equation}\label{eq:sum}
	\begin{split}
&n\Bigg\{\<\pi_t^n, H_t\>-\<\pi_0^n, H_0\>-\int_0^t\<\pi_s^n, \partial_sH_s\>\,ds\\
&+\big(G_t(\tfrac{1}{n})-H_t(\tfrac{1}{n})\big)\tfrac{1}{n}J^n_{0,1}(t)-\int_0^t(\p_sG_s(\tfrac{1}{n})-\p_sH_s(\tfrac{1}{n}))\tfrac{1}{n}J^n_{0,1}(s)\,ds\\&+\big(H_t(\tfrac{n-1}{n})-G_t(\tfrac{n-1}{n})\big)\tfrac{1}{n}J^n_{n-1,n}(t)-\int_0^t(\p_sH_s(\tfrac{n-1}{n})-\p_sG_s(\tfrac{n-1}{n}))\tfrac{1}{n}J^n_{n-1,n}(s)\,ds\Bigg\}\,.
\end{split}
\end{equation}
On the other hand,  the
integral term on the Radon-Nikodym derivative \eqref{der_rN} is given by
\begin{align*}
&
\int_0^t\big[\lambda(X_s)-\overline{\lambda}(X_s)\big]ds\;=\;
n\Bigg\{-\int_0^t \<\pi^n_s, \Delta_nH_s\>\,ds -\int_0^t \<\chi^n_s, (\nabla_n^+ H_s)^2\>ds\\&+\int_0^t\Big[ \eta_s(n-1)\nabla_n^- H_s(\pfrac{n-1}{n})- \eta_s(1)\nabla_n^+ H_s(\pfrac{1}{n})\Big]\,ds +O_H(\pfrac{1}{n})\\
&+ \sum_{x\in\{1,n-1\}}\!\!n^{1-\theta}\int_0^t \Big[r_x(1-e^{G_s(\frac{x}{n})})(1-\eta_s(x))+(1-r_x)(1-e^{-G_s(\frac{x}{n})})\eta_s(x)\Big]\,ds
\Bigg\}\,,
\end{align*}
where the discrete derivatives  $\nabla_n^+ H$ and $\nabla_n^- H$ and the discrete Laplacian $\Delta_nH$ have been defined in \eqref{nabla} and \eqref{lapla} and
\begin{equation}\label{eq:chi_rep}
\chi^n_s(du)\;=\;\frac{1}{2n}\sum_{x=1}^{n-2}\Big(\eta_s(x)-\eta_s(x+1)\Big)^2\delta_{\frac{x}{n}}(du)\,.
\end{equation}
Putting all together, the Radon-Nikodym  derivative is  given by
	\begin{equation}\label{correct_RN}
\begin{split}
& \frac{\dradon\bb P_{\subm}}{\dradon\bb P^{H,G}_{\subm}}\Bigg|_{\mc F_t} \;=\;  \exp \Bigg\{-n\;\Bigg[\<\pi^n_t,H_t\>-\<\pi^n_0,H_0\>
-\int_0^t \<\pi^n_s, (\p_s+\Delta_n)H_s\>\,ds\\
& -\int_0^t \<\chi^n_s, (\nabla^+_n H_s)^2\>ds+\int_0^t\Big[ \eta_s(n-1)\nabla_n^- H_s(\pfrac{n-1}{n})- \eta_s(1)\nabla_n^+ H_s(\pfrac{1}{n})\Big]\,ds+ O_H(\pfrac{1}{n}) \\
&- \sum_{x\in\{1,n-1\}}\!\!n^{1-\theta}\int_0^t \Big[r_x(e^{G_s(\frac{x}{n})}-1)(1-\eta_s(x))+(1-r_x)(e^{-G_s(\frac{x}{n})}-1)\eta_s(x)\Big]\,ds\\
&+\big(G_t(\tfrac{1}{n})-H_t(\tfrac{1}{n})\big)\tfrac{1}{n}J^n_{0,1}(t)-\int_0^t(\p_sG_s(\tfrac{1}{n})-\p_sH_s(\tfrac{1}{n}))\tfrac{1}{n}J^n_{0,1}(s)\,ds\\&+\big(H_t(\tfrac{n-1}{n})-G_t(\tfrac{n-1}{n})\big))\tfrac{1}{n}J^n_{n-1,n}(t)-\int_0^t(\p_sH_s(\tfrac{n-1}{n})-\p_sG_s(\tfrac{n-1}{n}))\tfrac{1}{n}J^n_{n-1,n}(s)\,ds \Bigg]\;\Bigg\}\,.
\end{split}
\end{equation}
At this point we impose that $G=H$, that is, we pick $G$ as  $G(\pfrac{1}{n})=H(\pfrac{1}{n})$ and $G(\pfrac{n-1}{n})=H(\pfrac{n-1}{n})$. The reason for such a choice  goes as follows for each regime of $\theta$.

For $\theta\in(0,1)$, as shown in Proposition~\ref{lem:rep_lem_sup} (see also \eqref{eq:V_theta} for $\theta\in(0,1)$),  the time integral of the occupation variables $\eta_s(1)$ and $\eta_s(n-1)$ can be replaced by $\alpha$ and $\beta$, respectively. This situation lies in the same scenario of \cite{flm} which works on the case $\theta=0$ and no perturbation over the current is required.  

 For $\theta\in(1,+\infty)$, we need a spoiler:   Lemma \ref{lem:current_exp} will assure that the normalized currents $\tfrac{1}{n}J^n_{n-1,n}$ and $\tfrac{1}{n}J^n_{0,1}$ are super-exponentially small. Hence, no perturbation at the boundary would contribute in the limit, and the choice $G=H$ takes place for sake of simplicity.
 
 Finally, we  justify why we did not start \textit{a priori} with the choice $G=H$.  First, for pedagogical reasons: the most natural form of the Radon-Nikodym derivative is given by \eqref{correct_RN}, including the current at the boundary. Second, but not less important, the critical case $\theta=1$ not treated here may require some perturbation at the boundary. Since we are seeking to investigate this in a future work, the general formulation is already presented here to avoid double-working.

Now, as usual, we replace the discrete Laplacian by the continuous Laplacian, the discrete derivative by the continuous derivative and the values of $H$ on $1/n$ and $(n-1)/n$ by the values of $H$ on $0$ and~$1$, respectively. These changes can be done by paying a price of  order $1/n$, because $H\in C^{1,2}$. Because of the choice $G=H$, the Radon-Nikodym derivative \eqref{correct_RN} can be rewritten as 
\begin{equation}\label{correct_RN_1}
\begin{split}
& \frac{\dradon\bb P_{\subm}}{\dradon\bb P^{H}_{\subm}}\bigg|_{\mc F_t} =  \exp \Bigg\{-n\;\Bigg[\<\pi^n_t,H_t\>-\<\pi^n_0,H_0\>
-\int_0^t \<\pi^n_s, (\p_s+\Delta)H_s\>\,ds\\
& -\int_0^t \<\chi^n_s, (\p_u H_s)^2\>ds+\int_0^t\Big[ \eta_s(n-1)\p_uH_s(1)- \eta_s(1)\p_u H_s(0)\Big]\,ds+O_H(\pfrac{1}{n}) \\
&- \sum_{x\in\{1,n-1\}}\!\!n^{1-\theta}\int_0^t \Big[r_x(e^{H_s(u_x)}-1)(1-\eta_s(x))+(1-r_x)(e^{-H_s(u_x)}-1)\eta_s(x)\Big]\,ds
\Bigg]\;\Bigg\}\,,
\end{split}
\end{equation}
where $u_1=0$ and $u_{n-1}=1$. 

Note that for $\theta\in(0,1)$, the last sum in \eqref{correct_RN_1} above may explode, motivating us to adittionally assume $H_s(0)=H_s(1)=0$ for all $s\in[0,T]$, also in agreement with \cite{flm}.
In Section~\ref{s5}  this Radon-Nikodym derivative will be further studied.

\subsection{Hydrodynamic limit for the perturbed process}
Recall the definition of the empirical measure from \eqref{eq:emp_mea}. Let $\mu_n$ be a measure in $\Omega_n$ associated with  a measurable  profile $\gamma(\cdot)$.   Denote by $\bb P_{\mu_{n}}^H$  the measure  on $\mc D([0,T],\mc M)$ induced by the Markov process with infinitesimal generator $n^2\mc L_{n}^{H,t}$ and the initial measure $\mu_n$ and denote by $\bb Q^H_{\mu_{n}} $ the probability on  $\mc D([0,T],\mc M)$ induced by
 $\{\pi^n_t;\; t\in[0,T]\}$ and the initial measure $\mu_n$.
 Recall the definition \eqref{C_theta_set} for ${\textbf{C}_\theta}$ and keep in mind that we additionally assume $H_s(0)=H_s(1)=0$ for all $s\in[0,T]$, when $\theta\in (0,1)$.
\begin{theorem}\label{thm:hid_lim_weak}
 Suppose that the 
sequence $\{\mu_n\}_{n\in \bb N}$ is \textit{associated} with a measurable  profile $\gamma(\cdot)$ in the sense of \eqref{eq3}. 
Then,  for each $ t \in [0,T] $, for any $ \delta >0 $ and any  function $ f \in C^0 $, 
\begin{equation*}
\lim_{ n \rightarrow +\infty }
\bb P_{\mu_{n}}^H \Big[\;\eta_{\cdot} : \Big\vert \<\pi^n_t,f\> - \<\rho^H_t,f\>\, \Big\vert
> \delta \;\Big] \;=\; 0\,,
\end{equation*}
 where  $\rho^H\in L^2(0,T;\mc H^1(0,1))$ and\medskip

$\bullet$ If $\theta\in(0,1)$,  then $\rho^H$  
  is the unique  solution of the integral equation
\begin{equation}\label{eq:weak_eq_theta_menor_um}
\begin{split}
 \FDir(t,f,\rho^H)\;:=\;&\<\rho^H_t,f_t\>-\<\gamma,f_0\>-\int_0^t \<\rho^H_s, (\p_s+\Delta)f_s\>\,ds \\&+\int_0^t \Big[\beta\, \p_uf_s(1) -\alpha\, \p_uf_s(0)\Big]\, ds   -2\int_0^t\<\chi(\rho^H_s) \,\p_uH_s, \p_uf_s\>\,ds\;=\;0\,, 
\end{split}
\end{equation} for all $t\geq 0$ and for all $f\in \C$.
\bigskip

$\bullet$ If $\theta\in(1,\infty)$, then  $\rho^H$  
  is the unique   solution of the integral equation
\begin{equation}\label{eq:weak_eq_theta_maior_1}
\begin{split}
 \FNeu(t,f,\rho^H)\;:=\;& \<\rho^H_t,f_t\>-\<\gamma,f_0\>-\int_0^t \<\rho^H_s, (\p_s+\Delta)f_s\>\,ds \\&+\int_0^t \Big[\rho^H_s(1)\p_uf_s(1) -\rho^H_s(0)\p_uf_s(0)\Big]\, ds   -2\int_0^t\<\chi(\rho^H_s) \,\p_uH_s, \p_uf_s\>\,ds \;=\;0\,, 
\end{split}
\end{equation} for all $t\geq 0$ and $f\in \C$. 
\end{theorem}
The classical counterpart of \eqref{eq:weak_eq_theta_menor_um} is  the partial differential equation
\begin{equation}\label{eq:edpasy_menor_um}
\left\{
\begin{array}{l}
\displaystyle \partial_t \rho \, =\, \Delta \rho -2\,\p_u\big(\chi(\rho)\p_u H \big) \\
 \rho_t (0)\,=\,\alpha\,, \quad \forall\,t\in(0,T]\\
 \rho_t (1)\,=\,\beta\,, \quad \forall\,t\in(0,T]\\
\displaystyle \rho(0,\cdot) \,=\, \gamma(\cdot)\\
\end{array}
\right.
\end{equation}
while the classical counterpart of \eqref{eq:weak_eq_theta_maior_1} is 
\begin{equation}\label{eq:edpasy=1_maior_um}
\left\{
\begin{array}{l}
\displaystyle \partial_t \rho \, =\, \Delta \rho -2\,\p_u\big(\chi(\rho)\p_u H \big) \\
\p_u \rho_t (0)\,=\,2\,\chi\big(\rho_t(0)\big)\,\p_u H_t(0)\,, \quad \forall\,t\in(0,T]\\
\p_u \rho_t (1)\,=\,2\,\chi\big(\rho_t(1)\big)\,\p_u H_t(1)\,, \quad \forall\,t \in(0,T]\\
\displaystyle \rho(0,\cdot) \,=\, \gamma(\cdot)\\
\end{array}
\right.
\end{equation}
that is, $\rho^H$ in each case is a weak solution of the respective PDE above. 
\begin{remark}\rm As the reader can observe,  the  PDE \eqref{eq:edpasy_menor_um} has Dirichlet boundary conditions, while the PDE \eqref{eq:edpasy=1_maior_um} has Robin boundary conditions. At a first glance, the fact that the PDE \eqref{eq:edpasy=1_maior_um} has Robin boundary conditions may look as a contradiction, since the corresponding PDE \eqref{hydroeq_Neumann} in the symmetric case  is of Neumann  type.  This apparent contradiction is due to the fact that such  PDE is not the heat equation, but \textit{the heat equation with a non linear drift}. By taking $f\equiv 1$ in \eqref{eq:weak_eq_theta_maior_1} we can see that the total mass of the solution $\rho$ of \eqref{eq:edpasy=1_maior_um} is time-invariant, which characterizes it as very close to the symmetric case with Neumann boundary conditions.   
\end{remark}

The outline  of the proof of Theorem \ref{thm:hid_lim_weak} goes as follows.  As usual, the proof is split into tightness of the sequence  $\lbrace\mathbb{Q}^H_{\mu_n}\rbrace_{n\geq 1}$ and the characterization  of  limit points of this sequence. Let us denote such a  limit point by $\mathbb Q^H$. By Prohorov's Theorem, the two last results  imply  the convergence of  $\lbrace\mathbb{Q}_{\mu_n}^H\rbrace_{n\geq 1}$ to $\mathbb Q^H$  as $n\rightarrow \infty$. 

 In Subsection~\ref{sec:tightness} we deal with the tightness issue, while in Subsection~\ref{sec:charac_lim_points} we characterize the limit point $\mathbb Q^H$ as having   density $\rho^H_{t}(\cdot)$ which is a weak solution of  the corresponding hydrodynamic equation. By the  uniqueness of weak solutions of the hydrodynamic equations proved in Subsection~\ref{subsec:uniqueness}, we conclude that $\lbrace\mathbb{Q}_{\mu_n}\rbrace_{n\geq 1}$ has a unique limit point $\mathbb{Q}$, which yields the convergence of the whole sequence  to that limit point $\mathbb Q^H$.

\subsection{Tightness}\label{sec:tightness}

In this section we show that the sequence of probability measures\break$\{\mathbb Q_{\mu_n}^H\}_{n\geq 1}$ is tight in the Skorohod space $\DM$.
By  \cite[Proposition 1.7,  Chapter 4]{kl} it is enough to show that for every  test function $f$ in a dense subset of $C^0$ with respect to the uniform topology, the sequence of measures that
corresponds to the real processes $\langle\pi_{t}^{n},f\rangle$ is tight.
The prove this last claim, we will use the  Aldous' Criterion, see \cite{aldous1978}.
\begin{lemma}[Aldous' Criterion] Let $(S,d)$ be a Polish metric space. 
A sequence $\{P_{n}\}_{n\geq 1}$ of probability measures defined on a Skorohod space $\mathcal{D}_{S}$ is tight if the two conditions below hold:
\begin{enumerate}[(a)]
\item \label{item_a}
For every $t\in{[0,T]}$ and every $\varepsilon>0$, there exists a compact set  $K_{\varepsilon}^{t}\subset{\mathcal{M}}$  such that
\begin{equation*}
\sup_{n\geq 1}P_{n}\Big(\zeta_{t}\notin{K_{\varepsilon}^{t}}\Big)\;\leq\; {\varepsilon}\,.
\end{equation*}
\item \label{item_b}
For every $\varepsilon>0$,
\begin{equation*}
\lim_{\gamma\downarrow{0}}\varlimsup_{n\rightarrow{\infty}}\sup_{\substack{\tau\in{\mathcal{T}_{T}}\\ \theta\leq{\gamma}}}
P_{n}\Big(d(\zeta_{\tau+\theta},\zeta_{\tau})>\varepsilon\Big)\;=\;0\,,
\end{equation*}
\end{enumerate}
where $\mathcal{T}_{T}$ denotes the set of stopping times with respect to the canonical filtration, bounded by $T$, and $\zeta_t$ denotes the value of $\zeta\in \mathcal{D}_S$ at time $t$.
\end{lemma}

The condition \eqref{item_a} above in our setting can be translated into
\begin{equation*}
\lim_{A\rightarrow{+\infty}}\varlimsup_{n\rightarrow{+\infty}}\mathbb{P}^H_{\mu_n}\Big(|\langle\pi_{t}^{n},f\rangle|>A\Big)\;=\;0
\end{equation*}
which follows from Chebychev's inequality and the fact there is at most one particle per site.
 Now we show  condition {(b)}, which in this context, asks that 
  for all $\varepsilon > 0$ and any function $f$ in a dense subset of $C^0$, with respect to the uniform topology,  
\begin{equation}
\label{eq:tight_1}
\displaystyle \lim _{\delta \downarrow 0} \varlimsup_{n\rightarrow\infty} \sup_{\tau  \in \mathcal{T}_{T},\bar\tau \leq \delta} {\mathbb{P}}^H_{\mu _{n}}\Big(\eta_{\cdot}:\left\vert \langle\pi^{n}_{\tau+ \bar\tau},f\rangle-\langle\pi^{n}_{\tau},f\rangle\right\vert > \varepsilon \Big) \;=\;0\,,
\end{equation}
where the  stopping times are bounded by $T$.

The verification of condition \eqref{item_b} in our setting requires  two different dense sets  with respect to $C^0$ in the uniform topology. Namely, the space $C^1$ for $\theta<1$ and the space $C^2$ for  $\theta\in(1,+\infty)$.  For  $\theta\in(0,1)$, we first prove tightness  for functions $f \in C^{2}_{c}$ and then we extend it by a $L^1$ approximation procedure which is explained in \cite{bmns} to functions $f\in C^1$.

Given $f:\Omega_n\to\bb R$, we know by Dynkin's formula (see Lemma~A1.5.1 of \cite{kl})  that
\begin{equation}\label{dynkin}
M^{n,H}_{t}(f)\; =\; \langle \pi^{n}_{t},f\rangle - \langle \pi^{n}_{0},f \rangle - \int_{0}^{t} (\partial_{s} + n^{2}\mc L_n^{H,s}) \langle \pi^{n}_{s},f \rangle \,ds
\end{equation}
is a martingale with respect to the natural filtration $\{\mathcal{F}_{t}\}_{t \geq 0} = \{\sigma(\eta_{s}): s \leq t \}_{t \geq 0}$. 
By a simple computation,  for $\eta\in\Omega_n$, for $x\in \Sigma_n$ and for $s\in[0,t]$, we have that $\mc L_n^{H,s} \eta_s(x)=j^{H_s}_{x-1,x}(\eta_s) -j^{H_s}_{x,x+1}(\eta_s)$, where the instantaneous  current $j^{H_s}_{x,x+1}(\eta_s)$  is given  by
\begin{equation}\label{eq:current_bulk}
 j^{H_s}_{x,x+1}(\eta_s) \;=\; e^{(\eta_s(x)-\eta_s(x+1))\frac1n \nabla^+_nH_s(\frac{x}{n})}\big(\eta_s(x)-\eta_s(x+1)\big) 
\end{equation}
at the bulk  $x\in \{1,\dots,n-2 \}$ and given by 
\begin{align}
j^{H_s}_{0,1}(\eta_s) & \;=\; \frac{1}{n^{\theta}}\big(e^{H_s(\frac{1}{n})}\alpha(1-\eta_s(1))-e^{-H_s(\frac{1}{n})}(1-\alpha)\eta_s(1)\Big)\,,\label{eq:current_boundary_1}\\
j^{H_s}_{n-1,n}(\eta_s) &\;=\;\frac{1}{n^{\theta}}\big(-e^{H_s(\frac{n-1}{n})}\beta(1-\eta_s(n-1))+e^{-H_s(\frac{n-1}{n})}(1-\beta)\eta_s(n-1)\Big)\,.\label{eq:current_boundary_n}
\end{align}
at the boundary. Moreover,  the martingale $M^{n, H}_{t}(f) $ can be rewritten as
\begin{equation}\label{eq:Dynkin_mart}
  \langle \pi^{n}_{t},f \rangle - \langle \pi^{n}_{0},f \rangle  - \int_{0}^{t} \sum_{x=1}^{n-2}\nabla^+_{n}f(\tfrac xn)j^{H_s}_{x,x+1}(\eta_s)ds  -\int_{0}^{t}\Big[ nf(\tfrac 1n)j_{0,1}^{H_s}(\eta_s)-nf(\tfrac {n-1}{n})j_{n-1,n}^{H_s}(\eta_s)\Big]ds\,. 
\end{equation}

We  start with the case  $\theta\in(1,+\infty)$ and   prove (\ref{eq:tight_1}) directly for functions  $f \in C^{2}$.  
By the triangular inequality and an union bound, the probability in \eqref{eq:tight_1} is equal or less than 
\begin{equation*}
{\mathbb{P}}^H_{\mu _{n}}\Big(\eta_{\cdot}: \Big| M_{\tau}^{n,H}(f)- M_{\tau+ \bar\tau}^{n,H}(f)  \Big| > \frac \varepsilon 2 \Big)
 +{\mathbb{P}}^H_{\mu _{n}}\Big(\eta_{\cdot}: \Big|  \int_{\tau}^{\tau+ \bar\tau} n^2 \mc L_{n}^{H,s} \langle \pi_{s}^{n},f \rangle ds  \Big| > \frac \varepsilon 2 \Big)\,.
\end{equation*}
Applying  Chebychev's inequality in the term on the left hand side  of last display and  Markov's inequality in the term on the right hand side  of last display, the proof ends as long as we show that
\begin{equation} 
\label{eq:tight_2}
\displaystyle\lim _{\delta \downarrow 0} \varlimsup_{n\rightarrow\infty} \sup_{\tau  \in \mathcal{T}_{T},\bar\tau \leq \delta}\mathcal{\mathbb{E}}^H_{\mu _{n}}\Big[ \Big| \int_{\tau}^{\tau+ \bar\tau}n^2 \mc L_{n}^{H,s}\langle \pi_{s}^{n},f\rangle ds \Big|\Big] \;=\; 0
\end{equation}
and
 \begin{equation} 
 \label{eq:tight_3}
\displaystyle \lim _{\delta \downarrow 0} \varlimsup_{n\rightarrow\infty}\sup_{\tau  \in \mathcal{T}_{T},\bar\tau \leq \delta}\mathcal{\mathbb{E}}^H_{\mu _{n}}\Big[\Big( M_{\tau}^{n,H}(f)- M_{\tau+ \bar\tau}^{n,H}(f) \Big)^{2}  \Big]\;=\;0
\end{equation}
where $\mathcal{\mathbb{E}}^H_{\mu _{n}}$ denotes the expectation with respect to $\bb P_{\mu_n}^H$.
Now we prove \eqref{eq:tight_2} and for that purpose recall \eqref{dynkin}, which is equal to  \eqref{eq:Dynkin_mart} as mentioned above.   A  computation, based on the Taylor expansion of the exponential function and the fact that $H\in C^{1,2}$, permits to rewrite
\begin{equation}
  \sum_{x=1}^{n-2}\nabla^+_{n}f(\tfrac xn)j^{H_s}_{x,x+1}(\eta_s) 
\end{equation}
 as 
\begin{equation*}
   \nabla_n^+f(0)\eta_s(1)-\nabla^-_nf(1)\eta_s(n-1)+ \frac{1}{n}\sum_{x=1}^{n-1}\Delta_{n}f(\tfrac xn)\eta_s(x)
\end{equation*} plus terms of order $O_f(1)$. Since $f\in C^2$ and the fact that the number of particles per site is at most one, the last expression is also  of order $O_f(1)$.

Now we analyse the boundary terms in \eqref{eq:Dynkin_mart}.  Since  $f\in C^2$, these terms are of order $O(n^{1-\theta})$. Since $\theta\in(1,+\infty)$ we conclude that 
$
n^2 \mc L_n ^H( \langle \pi^n_{s}, f \rangle )$ is  bounded by a constant. 
Note that for  $\theta\in(0,1)$, since we consider  $f \in C_c^{2}$, all the boundary terms that appear in the expression for $n^2 \mc L_{n} ^{H,s}( \langle \pi^n_{s}, f \rangle )$ vanish and the previous bound also shows  \eqref{eq:tight_2} for the case $\theta\in(0,1)$, provided the test functions are in $C^2_c$.

 Now we prove  \eqref{eq:tight_3}. The quadratic variation of the martingale $M^{n,H}_t$ is given by 
\begin{equation*}
\langle M^{n,H}(f)\rangle_t\;=\;\int_0^t \Big[n^2\mc L_n^{H,s} \langle \pi^{n}_{s},f_s \rangle^2-2\langle \pi^{n}_{s},f_s \rangle n^2\mc L_n^{H,s}  \langle \pi^{n}_{s},f_s \rangle \Big]\, ds
\end{equation*}
and some computations give us that 
the contribution from the bulk dynamics in the previous expression writes as 
\begin{equation}\label{eq:quad_var_1}
\begin{split}
\int_0^t\frac{1}{n^2}\sum_{x=1}^{n-1}\Big(\nabla_n^+ f(\tfrac xn )\Big)^2\Big(& e^{\frac1n \nabla^+_nH_t(\tfrac{x}{n})}\eta_s(x)(1-\eta_s(x+1))+e^{-\frac1n \nabla^+_nH_t(\tfrac{x}{n})}\eta_s(x+1)(1-\eta_s(x))\Big)ds
\end{split}
\end{equation}
and the contribution from the boundary dynamics writes as 
\begin{equation}\label{eq:quad_var_2}
\int_0^t\frac{1}{n^\theta} \sum_{x\in\{1,n-1\}} f^2(\tfrac xn )\,\Big[ e^{H_t(\frac xn)}r_x(1-\eta(x)) +e^{-H_t(\frac xn)}(1-r_x)\,\eta(x)\Big]ds\,.
\end{equation}
Since $H\in C^{1,2}$, $f\in C^2$ and the fact that there is at most a particle per site, we conclude that the quadratic 
variation of the martingale $M^{n,H}_t$ is of order $O(\tfrac{1}{n}+\tfrac {1}{n^{\theta}})$, which vanishes  as $n\to+\infty$. Since $C^2$ is a dense subset of $C$, with respect to the uniform topology, the proof of tightness in the case $\theta\in(1,+\infty)$ ends. Now let us go back to the case $\theta\in(0,1)$. Recall that we have already seen above that for test functions in $C^2_c$ the limit in \eqref{eq:tight_2} is true. It remains to show \eqref{eq:tight_3}. But as in the case $\theta\in(1,+\infty) $ we can conclude that the quadratic variation of the  corresponding martingale is of order $O(\tfrac{1}{n})$ and again it vanishes  as $n\to+\infty$. This ends the proof of tightness.

\subsection{Replacement lemmas and energy estimates}

In this section we state  the replacement lemmas that we need in order to recognize the density profile as a weak solution of the corresponding hydrodynamic equation. At the end of this section we prove that the profile belongs to the Sobolev space given  in Definition \ref{Sobolev}.
We start with a replacement lemma which suits all cases of $\theta$. Recall \eqref{eq:V_0} and \eqref{eq:V_theta}.  In what follows $\varphi\in C^{0,0}$. 

 \begin{lemma} \label{lem:rep_lem_general}
For any $t\in[0,T]$, for any $\theta$  and  for   $ x=0,1,n-1$ we have that
\begin{equation*}
\begin{split}
&\varlimsup _{\varepsilon \downarrow 0}\varlimsup _{n\rightarrow \infty} \bb E_{\mu_n}\Big[\Big|\int_0^tV_{\eps,x}^{\theta,\varphi^n}(\eta_s,s)\, ds\Big|\Big] \;=\;0\,.
\end{split}
\end{equation*}
\end{lemma}

From  the super-exponential replacement lemma stated in Lemma \ref{lem:rep_lem_sup} together the fact that the Radon-Nikodym derivative is bounded and an entropy estimate (needed in order to change measures), we obtain all the replacement lemmas stated above.  For this reason we omit their proofs and  leave the gaps to the  reader. 
Finally, we note that the density $\rho^H_{t}(u)$ belongs to $L^{2}(0,T;\mathcal{H}^{1})$, see  Definition \ref{Sobolev}. For that purpose, let us define the linear functional $\ell_{\rho^H}$ on $C^{0,1}_{c}$ by 
\begin{equation*}
\ell_{\rho^H}(f) \;=\; \<\!\< \partial_{u}f,\rho^H\>\!\> \;=\; \int^{T}_{0}\int_0^1 \partial_{u}f_{s}(u)  \pi_{s}(du) ds\,.
\end{equation*}

\begin{lemma} 
\label{EE}
The following inequality holds:
\begin{equation*}
\mathbb{E}\left[ \sup _{f\in C^{0,1}_{c} }\Big\lbrace \ell_{\rho^H}(f) - 2 \Vert f \Vert_{L^2(0,T;(0,1))}^{2}\Big\rbrace \right]\;\lesssim\; 1\,.
\end{equation*}
\end{lemma}
From  the last result  it follows that  $\ell_{\rho^H}$ is $\bb Q^H$ almost surely continuous, so that this linear functional  can be extended to $L^{2}([0,T]\times (0,1))$. Then, by the Riesz's Representation Theorem,  we can find $\zeta \in L^{2}([0,T]\times (0,1))$ such that
$\ell _{\rho^H}(f) = -\<\!\< f,\zeta\>\!\>$
for all $f \in C^{0,1}_{c}$, which implies   $\rho^H \in L^{2}(0,T;\mathcal{H}^{1})$.

\subsection{Characterization of limit points}\label{sec:charac_lim_points}
Since at most one particles is allowed per site,   any limit point of  the sequence $\lbrace\mathbb{Q}^H_{\mu_n}\rbrace_{n\geq 1}$ is concentrated on trajectories of  measures that are  absolutely continuous with respect to the Lebesgue measure. That is, any limit point  $\bb Q^H$ of the sequence sequence $\{\mathbb Q_n^H\}_{n \ge 1}$ is concentrated on trajectories of measures $\pi_{t}(du)$  such that $\pi_{t}(du)=\rho^H(t,u)du$.

 Since the initial measure is associated to the profile $\gamma(\cdot)$ we also know  that all limit points $\bb Q^H$ of the sequence $\{\mathbb Q^H_{\mu_n}\}_{n \ge 1}$ are concentrated on the initial measure $\pi_{0}(du)=\gamma(u)du$. Now we prove that all limit points are concentrated on  trajectories of measures of the form $\rho^H_t(u)du$, where  $\rho^H_{t}(\cdot)$ is a weak solution of the corresponding hydrodynamic equation. 
  For that purpose, let $\bb Q^H$ be a limit point of the sequence $\lbrace\bb Q_{\mu_n}^H\rbrace_{n \geq 1}$ and assume, without loss of generality that $\lbrace\bb Q_{\mu_n}^H\rbrace_{ n \ge 1}$ converges weakly to $\bb Q^H$ as $n\to+\infty$.  

\begin{proposition}\label{prop:weak_sol}
If $\bb Q^H$ is a limit point of  $ \{\bb Q^H_{\mu_n}\}_{n\in\mathbb N}$,  then 
\begin{equation*}
\bb Q^H\Big( \pi\in \DM \;:\; \pi_t(du)=\rho_t(u)du \text{ and }\mc F_\theta(t, f,\rho)= 0,\forall t\in [0,T]\,,\, \forall f \in \C\,\Big)\;=\;1\,,
\end{equation*}
where $\C$ has been defined in \eqref{C_theta_set}, and
\begin{equation*}
\mc F_\theta(t, f,\rho) \;:=\; 
\begin{cases}
\FDir(t,f,\rho), \;\textrm{if }\theta\in[0,1),\\
\FNeu(t,f,\rho), \; \textrm{if }\theta\in(1,+\infty),\\
\end{cases}
\end{equation*}
with $\FDir$ and $\FNeu$ defined in  \eqref{eq:weak_eq_theta_menor_um} and \eqref{eq:weak_eq_theta_maior_1}.
\end{proposition}
\begin{proof} 
Let us start  with the case  $\theta\in (1,+\infty)$.
It is enough to check that, for any $\delta > 0$ and any $ f\in \C= C^{1,2}$,   
\begin{equation}\label{prob_charac}
\bb Q^H\bigg(\pi\in \DM \;:\; \sup_{0\le t \le T} \left\vert \FNeu(t,f) \right\vert>\delta\bigg)\;=\;0\,.
\end{equation}
For $u\in[0,1]$ and $\varepsilon>0$, let $\iota_\varepsilon(u): [0,1]\to \bb R$ be an approximation of the identity defined as 
\begin{equation}\label{iota}
\iota_\varepsilon(u)(v)\;:=\; \begin{cases} \varepsilon^{-1}  \; \mathbf{1}_{(u,u+\varepsilon)}(v),  & \text{ if } u \in [0, 1 - 
       \varepsilon), \\
\varepsilon^{-1}\; \mathbf{1}_{(u-\varepsilon,u)}(v), & \text{ if } u \in (1 -        \varepsilon, 1].
 \end{cases}
\end{equation}
Note that $\eta^{\eps n}_s(x)=\pi^n_s*\iota_\eps(\tfrac{x}{n})$ and
 \begin{equation}\label{convolutions}
 \pi_s* \iota_\varepsilon(u)\;:=\; \begin{cases} \tfrac{1}{\varepsilon}\int_{u}^{u+\varepsilon}\rho^H_{s}(v)dv,  & \text{ if } u \in [0, 1 - 
       \varepsilon), \\
 \tfrac{1}{\varepsilon}\int_{u-\varepsilon}^u \rho^H_{s}(v)dv, & \text{ if } u \in (1 - \varepsilon, 1],
        \end{cases}
\end{equation}
since $\bb Q^H$ is concentrated on trajectories of measures that are    absolutely continuous with respect to the Lebesgue measure, that is, $\pi_{t}(du)=\rho^H_t(u)du$.
By adding and subtracting   $\pi_s* \iota_\varepsilon(0)$ and $\pi_s* \iota_\varepsilon(1)$ to $\rho^H_{s}(0)$  and to  $\rho^H_{s}(1)$,
 respectively, by adding and subtracting  $\chi(\pi_s* \iota_\varepsilon(u))$ to    $\chi(\rho^H_s(u))$, and applying the triangular inequality, 
 we can now bound the probability in \eqref{prob_charac} by  the sum of the following probabilities:
\begin{align}\label{eq:char_1}
\bb Q^H\bigg(& \pi_t(du)=\rho(du)\;:\; \sup_{0\le t \le T} \Big|\<\rho^H_t,f_t\>-\<\gamma,f_0\>-\int_0^t \<\rho^H_s, (\p_s+\Delta)f_s\>\,ds  \\
&- \int_0^t2\<\chi(\pi_s* \iota_\varepsilon) \,\p_uH_s, \p_uf_s\>\,ds +\int^{t}_{0} \Big[\pi_s* \iota_\varepsilon(1)\p_uf_s(1) -\pi_s* \iota_\varepsilon(0)\p_uf_s(0)\Big] \, ds \Big|>\dfrac{\delta}{3}\bigg)\,,\notag\\
\bb Q^H\bigg(& \pi_t(du)=\rho(du)\;:\;  \Big| \int_0^t2\<(\chi(\rho^H_s)-\chi(\pi_s* \iota_\varepsilon)) \,\p_uH_s, \p_uf_s\>\,ds \Big|>\dfrac{\delta}{3}\bigg)\,,\label{eq:char_2}\\
\bb Q^H \bigg(& \pi_t(du)=\rho(du)\;:\; \sup_{0\le t \le T} \int^{t}_{0} \Big[( \rho^H_{s}(1)-\pi_s* \iota_\varepsilon(1))\p_uf_s(1) -( \rho_{s}^H(0)-\pi_s* \iota_\varepsilon(0))\p_uf_s(0)\Big] \,ds    \Big|>\dfrac{\delta}{3}\bigg).\label{eq:char_3}
\end{align}
Now to control \eqref{eq:char_2}, observe that, by the triangular inequality and the fact that $\rho^H_s(\cdot)\leq 1$ for all $s\in[0,T]$, we have that 
\begin{equation}
\begin{split}
\Big|\chi(\rho^H_s(u)-\chi(\pi_s* \iota_\varepsilon)(u) \Big|\;\leq\; C |\rho^H_s(u)-\pi_s* \iota_\varepsilon(u)|
\end{split}
\end{equation}
and from Lebesgue's differentiation theorem last expression vanishes as $\varepsilon\to 0$, for a.e. $u\in[0,1]$. In a similar way, in order  to control  \eqref{eq:char_3},  we just need to use the fact that $\rho^H\in L^{2}(0,T;\mathcal H^1)$, to show that,  for $j\in\{0,1\}$
\begin{equation}
\begin{split}
\lim_{\varepsilon\to 0}\Big|\rho^H_s(j)-(\pi_s* \iota_\varepsilon)(j) \Big|\;=\;0\,.
\end{split}
\end{equation}

 Since $\bb Q^H$ is the weak limit of $\{\bb Q^H_{\mu_n}\}_{n\in \bb N}$, we would like to apply Portmanteau's Theorem to deal with \eqref{eq:char_1}. However,    the function $\iota_\varepsilon$ is not continuous, so this  is, in principle, not possible. However, as in   \cite[Proposition A.3]{fgn1}, by  approximating $\iota_\varepsilon$ by a continuous function,  in such a way that the error vanishes as $\varepsilon \to 0$, we can   bound \eqref{eq:char_1} from above by
\begin{equation}\label{eq:char_5}
\begin{split}
\varliminf_{n\to+\infty}& \bb Q_{\mu_n}^H \bigg( \pi_t(du)=\rho(du)\;:\;  \sup_{0\le t \le T} \Big|\<\rho^H_t,f_t\>-\<\gamma,f_0\>-\int_0^t \<\rho^H_s, (\p_s+\Delta)f_s\>\,ds  \\
-&\int^{t}_{0} 2\<\chi(\pi_s* \iota_\varepsilon) \,\p_uH_s, \p_uf_s\>\,ds  \,ds+ \int^{t}_{0} \Big[\pi_s* \iota_\varepsilon(1)\p_uf_s(1) -\pi_s* \iota_\varepsilon(0)\p_uf_s(0)\Big] \,ds  \Big|>\dfrac{\delta}{3}\bigg)\,,
\end{split}
\end{equation}
plus a term that vanishes as $\eps\to 0$.
Now we make use of the martingale \eqref{dynkin}. Recal that $\bb Q^H_{\mu_n}$ is induced by $\bb P^H_{\mu_n}$ and the empirical measure $\pi$, that is, $\bb Q^H_{\mu_n}=\bb P^H_{\mu_n} \circ \pi^{-1}$. By adding  and subtracting $\int_{0}^{t} n^{2}\mc L_{n}^{H,s}\langle \pi_{s}^{n},f_{s}\rangle ds$ to the term inside last probability,  we can bound \eqref{eq:char_5}  from above by the sum of 
\begin{equation}\label{eq:char_6}
\varliminf_{n\to\infty}\,\bb P_{\mu_n}^H\left( \sup_{0\le t \le T} \left\vert \mc M_{t}^{n,H}(f) \right\vert>\dfrac{\delta}{6}\right)\,,
\end{equation} 
and
\begin{equation}
\label{eq:char_7}
\begin{split}
\varliminf_{n\to\infty}\,\bb P_{\mu_n}^H&\bigg(  \sup_{0\le t \le T} \Big|\int_0^t n^{2}\mc L_{n}^{H,s}\langle \pi_{s}^{n},f_{s}\rangle\,ds-\int_0^t \<\rho^H_s, \Delta f_s\>\,ds   \\
-&\int_0^t2\<\chi(\pi_s* \iota_\varepsilon) \,\p_uH_s, \p_uf_s\>\,ds+ \int^{t}_{0} \Big[\eta_s^{\varepsilon  n}(n-1)\p_uf_s(1) -\eta_s^{\varepsilon  n}(1)\p_uf_s(0)\Big] \,ds   \Big|>\dfrac{\delta}{6}\bigg)\,.
\end{split}
\end{equation}
 By using Doob's inequality together with \eqref{eq:quad_var_1} and \eqref{eq:quad_var_2}, it is easy to show   that (\ref{eq:char_6}) vanishes  as $n\to\infty$. Now, \eqref{eq:char_7} can be rewritten as 
 \begin{equation}
\label{eq:char_7b}
\begin{split}
\varliminf_{n\to\infty}\,\bb P_{\mu_n}^H&\left(  \sup_{0\le t \le T} \Big|\int_0^t n^{2}\mc L_{n}^{H,s}\langle \pi_{s}^{n},f_{s}\rangle\,ds-\int_0^t \langle \pi_{s}^{n} ,\Delta f_{s}\rangle\,ds \right. \\
-&\left.\int_0^t2\<\chi(\pi^n_s* \iota_\varepsilon) \,\p_uH_s, \p_uf_s\>\,ds+ \int^{t}_{0} \Big[\eta_s^{\varepsilon  n}(n-1)\p_uf_s(1) -\eta_s^{\varepsilon  n}(1)\p_uf_s(0)\Big] \,ds   \Big|>\dfrac{\delta}{6}\right)\,.
\end{split}
\end{equation} 
From the computations right below \eqref{dynkin}, we have that
\begin{equation}\label{eq:gen_action}
\begin{split}
n^{2}\mc L_{n}^{H,s}\langle \pi_{s}^{n},f_{s}\rangle\;= \; - nf_s(\tfrac 1n)j_{0,1}^{H_s}+nf_s(\tfrac {n-1}{n})j_{n-1,n}^{H_s} +\sum_{x=1}^{n-2}\nabla_{n}^+f_s(\tfrac xn)j^{H_s}_{x,x+1}(\eta_s). 
\end{split}
\end{equation}
Recall \eqref{eq:current_bulk}. By doing a Taylor expansion on the exponential  in $j_{x,x+1}^{H_s}$,  the  term on the right hand side of last expression is equal to
\begin{equation*}
\begin{split}
\sum_{x=1}^{n-2}\nabla_{n}^+f_s(\tfrac xn)(\eta_s(x)-\eta_s(x+1))+\frac 1n\sum_{x=1}^{n-2}\nabla_{n}^+f_s(\tfrac xn)(\eta_s(x)-\eta_s(x+1))^2\nabla_n^+H(\tfrac xn)
\end{split}
\end{equation*}
plus a term of order $O_H(\tfrac 1n)$. A summation by parts shows that the  term on the right hand side of last expression can  be written as
\begin{equation*}
\nabla_n^+f_s(0)\eta_s(1)-\nabla^+f_s(\tfrac{n-1}{n})\eta_s(n-1)+\frac{1}{n}\sum_{x=1}^{n-1}\Delta_{n}f_s(\tfrac xn)\eta_s(x)\,.
\end{equation*}

Then, we can bound from above the probability in (\ref{eq:char_7}) by the sum of the following terms
\begin{equation}
\label{eq:char_8}
\bb P_{\mu_{n}}^H  \left(\sup_{0\le t \le T} \Big| \int_{0}^{t} \Big(\frac{1}{n}\sum_{x=1}^{n-1}\Delta_{n}f_s(\tfrac xn)\eta_s(x) -\left\langle \pi_{s}^{n},\Delta f_{s} \right\rangle\Big)   \, ds \Big|>\dfrac{\delta}{24}\right),
\end{equation}
\begin{equation}
\label{eq:char_9}
\begin{split}
\bb P_{\mu_{n}}^H  \Big( \sup_{0\le t \le T}  \Big| \int_{0}^{t} \Big(\frac 1n\sum_{x=1}^{n-2}\nabla_{n}^+f_s(\tfrac xn)(\eta_s(x)-\eta_s(x+1))^2\nabla_n^+H_s(\tfrac xn) -2\<\chi(\pi^n_s* \iota_\varepsilon) \,\p_uH_s, \p_uf_s\>\Big)\,ds  \Big|>\dfrac{\delta}{24}\Big),
\end{split}
\end{equation}
\begin{equation}
\label{eq:char_10}
\begin{split}
&\bb P_{\mu_{n}}^H  \left( \sup_{0\le t \le T}  \Big| \int_{0}^{t}\Big( \nabla _n^+f_s(0)\eta_{s}(1)- {\eta}^{\varepsilon n}_{s}(1) \partial_{u}f_{s}(0) \Big)  \, ds  \Big|>\dfrac{\delta}{24}\right),
\end{split}
\end{equation}
plus terms which are very similar to the previous one but related to the  action of the right boundary dynamics, plus other terms that vanish as $n \to +\infty$ due to the fact that $f\in C^{1,2}$.  Now, the proof ends by doing the following arguments. 
From   Taylor expansion on $f_s$ we easily treat the probability in \eqref{eq:char_8}. From  Taylor expansion on both $f_s$ and $H_s$, together with Markov's inequality and  Lemma \ref{lem:rep_lem_general} for the case $x=0$, for $\varphi^n=\partial_uf_s\partial_uH_s$ and $u_x=x/n$ we are able to treat the probability in \eqref{eq:char_9}. Finally, to treat the probability in \eqref{eq:char_10}, we just need to apply Taylor expansion to $f_s$, together with Markov's inequality and  Lemma \ref{lem:rep_lem_general} for the case $\theta\in(1,+\infty)$,  $x=1$, for $\varphi^n=\partial_uf_s$ and $u_x=0$. We leave the details to the reader.

Now we do the sketch of the characterization of limit points in the case $\theta\in(0,1)$. In this case  $f\in C^{1,2}_0$ and $\mathcal F_{\textrm{Dir}}$ was defined in \eqref{eq:weak_eq_theta_menor_um}. Since  $\mathcal F_{\textrm{Dir}}$ and $\mathcal F_{\textrm{Neu}}$ have a very similar expression, the only difference in the proof now is that  the boundary term  in
 \eqref{eq:char_3}
is  replaced by
\begin{align}
\bb Q^H \bigg(& \sup_{0\le t \le T} \int^{t}_{0} \Big[( \beta-\pi_s* \iota_\varepsilon(1))\p_uf_s(1) -( \alpha-\pi_s* \iota_\varepsilon(0))\p_uf_s(0)\Big] \,ds    \Big|>\dfrac{\delta}{3}\bigg).\label{eq:char_3_dir}
\end{align}
All the other terms can be treated exactly as we did in the case $\theta\in(1,+\infty)$. Now, in order to control the last probability we just need to apply Markov's inequality and  Lemma \ref{lem:rep_lem_general} for the case $\theta\in(0,1)$,  $x=1$, for $\varphi^n=\partial_uf_s$ and $u_x=0$. We leave the details to the reader.
\end{proof}
\subsection{Uniqueness of weak solutions}\label{subsec:uniqueness}
In this subsection we assure  uniqueness of weak solutions of  equations  \eqref{eq:weak_eq_theta_menor_um}  and  \eqref{eq:weak_eq_theta_maior_1}.
These proofs are based on the fact that the eigenfunctions of the Laplacian with Neumann (and with 
Dirichlet) boundary conditions are an orthonormal basis. 
Recall that if $\{ \Psi_k\}_k$ is an orthonormal basis of $L^2[0,1]$, then  for all $f\in L^2$, 
\begin{equation}\label{UN1}
\int f^2\,du\;=\;\sum_{k\geq 0}\<f,\Psi_k\>^2\,.
\end{equation}

\subsubsection{The Neumann case: \texorpdfstring{$\theta\in(1,+\infty)$}{theta<1}}
Let $\rho^1$ and $\rho^2$ be weak solutions of \eqref{eq:weak_eq_theta_maior_1} such that $\rho^1_0=\gamma=\rho^2_0$. Denote $\overline{\rho}=\rho^1-\rho^2$ and consider the set $\{\psi_k\}_{k\geq 0}$ of eigenfunctions of Laplacian with Neumann boundary conditions, i.e., $\psi_k(u)=\sqrt{2}\cos(k\pi u)$ for $k\geq 1$ and $\psi_0(u)=1$, which is, in fact,  an orthonormal basis of $L^2([0,1])$.
Now, define
\begin{equation*}\label{UN11}
\mc R(t)\;=\; \sum_{k\geq 0}\frac{1}{2c_k}\,\< \overline{\rho}_t , \psi_k\>^2\,,
\end{equation*}
where $c_k=(k\pi)^2+1$.
Our goal here is to prove that 
\begin{equation}\label{UN2}
\mc R'(t)\;\lesssim \;\mc R(t)\,.
\end{equation}
because, provided by this inequality,  Gronwall's inequality permits to conclude that $R'(t)\leq 0$, which leads  to $\rho^1=\rho^2$ a.e.
To achieve our goal, we start by computing the derivative of $\mc R$, which is given by
\begin{equation}\label{UN3}
\mc R'(t)\;=\; \sum_{k\geq 0}\frac{1}{c_k}\,\< \overline{\rho}_t , \psi_k\>\,\frac{d}{dt}\< \overline{\rho}_t , \psi_k\>\,.
\end{equation}
Using the integral equation  \eqref{eq:weak_eq_theta_maior_1}, the expression $\pfrac{d}{dt}\< \overline{\rho}_t , \psi_k\>$ in the last display above  is equal to
\begin{equation*}
\<\bar\rho_t,\Delta\psi_k\>\,+\,2\,\<\overline{\chi}\,\p_uH_t,\p_u \psi_k\>\,,
\end{equation*}
where $\overline{\chi}=\chi(\rho^1_t)-\chi(\rho^2_t)$. Note that $\<\overline{\rho}_t,\Delta\psi_k\>=-(k\pi)^2\<\overline{\rho}_t,\psi_k\>$.
Plugging this into \eqref{UN3}, we get
\begin{equation}\label{UN4}
\mc R'(t)\;=\; -\sum_{k\geq 0}\frac{(k\pi)^2}{c_k}\,\< \overline{\rho}_t , \psi_k\>^2+
\sum_{k\geq 0}\frac{2}{c_k}\,\< \overline{\rho}_t , \psi_k\>\<\overline{\chi}\p_uH_t,\p_u \psi_k\>\,.
\end{equation}
Now, Young's inequality allows to bound the previous expression by
\begin{equation}\label{UN5}
\frac{1}{A}\sum_{k\geq 0}\frac{1}{c_k}\,\< \overline{\rho}_t , \psi_k\>^2+A\sum_{k\geq 0}\frac{1}{c_k}\<\overline{\chi}\p_uH_t,\p_u \psi_k\>^2\,,
\end{equation}
where the specific value $A>0$ will be chosen later.
Now, observe that $\p_u \psi_k(u)=-k\pi\,\varphi_k(u)$ with $\varphi_k(u)=\sqrt{2}\sin(k\pi u)$ for $k\geq 1$ and $\varphi_0(u)=1$.
Therefore we can bound the second  sum in the display  by
\begin{equation*}
\sum_{k\geq 0}\frac{(k\pi)^2}{c_k}\<\overline{\chi}\,\p_uH_t,\varphi_k\>^2\;\leq\;\sum_{k\geq 0}\<\overline{\chi}\,\p_uH_t,\varphi_k\>^2\,,
\end{equation*}
because $c_k=(k\pi)^2+1$. Since $\{\varphi_k\}_{k\geq 0}$ is an orthonormal basis of $L^2[0,1]$, it is possible to use \eqref{UN1} to write the last sum as 
$\int_0^1\big(\overline{\chi}\p_uH_t\big)^2\,du$. Using the definition of $\overline{\chi}$ and the fact that $\chi$ is a Lipschitz function, we have
$\int_0^1\big(\overline{\chi}\,\p_uH_t\big)^2\,du\,\leq\,C_H\int_0^1\big(\overline{\rho}_t\big)^2\,du.$
Then using again \eqref{UN1} to rewrite  $\int \big(\overline{\rho}_t\big)^2du$ as 
$\sum_{k\geq 0}\< \overline{\rho}_t, \psi_k\>^2$, we get that  
\begin{equation*}
\mc R'(t) \;\leq\;\sum_{k\geq 0}\Big(-\frac{(k\pi)^2}{c_k}+\frac{1}{Ac_k}+C_HA\Big)\,\< \overline{\rho}_t , \psi_k\>^2\,.
\end{equation*}
Now choosing $A=\frac{1}{C_H}$ we finally get \eqref{UN2}.

\subsubsection{The Dirichlet case: \texorpdfstring{$\theta\in(0,1)$}{theta>1}}
This proof in this case is  similar to the one	  above, considering the set $\{\psi_k\}_{k\geq 0}$ of eigenfunctions of the Laplacian with Dirichlet boundary conditions, where $\psi_k(u)=\sqrt{2}\sin(k\pi u)$. Details are omitted here.
\section{Large deviations upper bound}\label{s5}
In this section we establish the large deviations uper bound, first for compact sets, then to closed sets. To do so, the following notion is relevant.
We say a family of sets $\{\Gamma_\lambda\}_\lambda$ is \textit{super-exponentially small} whenever
\begin{equation*}
\varlimsup_{\lambda}\frac{1}{\lambda}\log P\big[\Gamma_\lambda\big]\;=\;-\infty
\end{equation*}
where the limsup in $\lambda$ (or more parameters) depends on the context.

Let us describe the line of ideas for the proof of the upper bound. 
By the perturbed model presented in Section~\ref{s4}, we have that
\begin{equation*}
\bb P_{\delta{\eta^n}}\big[\{\pi^n\in\mc C\}\cap\mc G\big]\;=\;\bb E_{\delta{\eta^n}}^H\bigg[\textbf{1}_{\{\pi^n\in\mc C\}\cap\mc G} \cdot\frac{\dradon\bb P_{\delta_{\eta^n}}}{\dradon\bb P^{H}_{\delta_{\eta^n}}}\bigg|_{\mc F_T} \bigg]\,,
\end{equation*}
where the Radon-Nikodym derivative above has been computed in \eqref{correct_RN_1} and the {\it{good set}} $\mc G$ will bedefined in \eqref{G_set} is a set such that its complement is super-exponentially small.
In  Subsection~\ref{R-N_der_cont} we consider this Radon-Nikodym derivative  restricted to the {\it{good set}} $\mc G$,  obtaining the expression  of  the large deviations rate functional. Finally, in Subsection \ref{sub_5.3} we prove the upper bound for compact sets, and in Subsection \ref{sub_5.4} we extend it to closed sets  by a standard argument on exponential tightness.

\subsection{Superexponentially small sets}\label{small_sets}
Define the set
\begin{equation}\label{B_set}
B_{\eps,\delta}^{H,\theta}\;:=\;\Big\{\eta_.\in \Ddiscreto\;:\;\Big\vert\int_0^TV_{\eps,x}^{H,\theta}(\eta_s,s)\,ds\Big\vert\leq\delta,\,\,x=0,1,n-1\Big\}\,,
\end{equation}
where $V_{\eps,x}^{H,\theta}(\eta_s,s)$, as defined in \eqref{eq:V_0} and \eqref{eq:V_theta}, is taken under  the particular choice\break  $\varphi_s^n(u_x)=\p_uH_s(\frac xn)$.
By Proposition~\ref{lem:rep_lem_sup}, we know that
\begin{equation}\label{B_set_lim}
 \varlimsup_{\eps\downarrow 0} \varlimsup_{n\to\infty}\tfrac 1n \log \bb P_{\subm}\Big[\big (B_{\eps,\delta}^{H,\theta}\big)^\complement\Big]\;=\;-\infty
\end{equation}
for all $\delta,\theta>0$ and $H\in {\bf{C}}_\theta$. Before introducing the next  super-exponential small set, which is somewhat technical, let us discuss its rather simple motivation. Keep in mind that our  objective is to 
asymptotically deal with the Radon-Nikodym derivative, which  will lead us to the large deviations rate functional.

Recall that $\eta^{\eps n}_s(x)=\pi^n_s*\iota_\eps(\tfrac{x}{n})$, where the approximation of the identity $\iota_\eps(u)(v)$ has been  defined in \eqref{iota}.
Although important,  the extra regularity given by this convolution is not enough to handle 
  limits at the boundaries, since, in general, $\pi*\ioe$ is not a continuous function. To overcome this, we shall (super-exponentially) replace  $\pi^N*\ioe$ by $(\pi^N*\iog)*\ioe$, where $\iog$ is a \textit{smooth} approximation of the identity that is defined as follows. 

Fix $f:[0,1]\to \bb R_+$ a 
continuous function with support  contained in $[\frac{1}{4},\frac 34]$,
$0\leq f \leq 4$, $f(0)>0$, $\int fd\lambda=1$ and symmetric around zero,  that is,  satisfying  $f(u)=f(1-u)$ for all $u\in [0,1]$. Define  the continuous approximation of identity
$\iog$  by
$
\iog(u)=\pfrac{1}{\tau} f(\pfrac{u}{\tau})$. 
As in Lemmas 5.1, 5.2 and 5.3 of \cite{FN2017}, changing $\pi^n*\ioe$ by $(\pi^n*\ioe)*\iog$ inside the expression of the Radon-Nikodym derivative has a cost of order $O_H(\eps)+O_H(\pfrac{\tau}{\eps})$.

Since the rate functional is equal to  infinite on trajectories $\pi\in \DM$ such that $\mc E(\pi)<\infty$, another important remark about the double convolution is that $\mc E((\pi*\iog)*\ioe)<\infty$ for all $\pi\in \DM$. 

 The next set what we introduce is the set that handles with trajectories with finite energy, that is,   the set $\{\pi\in \DM\;;\;\mc E(\pi)<\infty\}$. Since this set is not closed with respect to the Skorohod topology of $\DM$, this is an obstacle to apply the \textit{Minimax Lemma} (see \cite{kl}, page 364, Lemma~3.3), which is an important device in the proof of the large deviations' upper
bound. To overcome this difficult, we introduce   the following sets.
Let $A_{k,l}$ and $A_{k,l}^{\zeta,\tau}$  be the subsets of trajectories  given by
\begin{equation}\label{A_set}
\begin{split}
&A_{k,l}\;=\;\{\pi\in \DM:\max_{1\leq j\leq k}\mc E_{H_j}(\pi)\leq l\}\,,
\\
&A_{k,l}^{\zeta,\tau}\;=\;\left\{\pi\in \DM:(\pi*\iog)* \ioz 	\in A_{k,l}\right\}\,.
\end{split}
\end{equation}
It is worth to emphasize that $\ioz$ is the identity approximation defined in \eqref{iota}, where the letter $\eps$ has been replaced by $\zeta$ for  aesthetic reasons.
For fixed $\zeta,\tau,k,l$, the set $A_{k,l}^{\zeta,\tau}$ is closed because the function $\pi\mapsto\mc E_{H}((\pi*\iog)*\ioz)$ is continuous in the Skorohod topology,  see \cite{FN2017} for instance.
We claim that, for fixed $k$ and $l$,
\begin{equation}\label{A_set_lim}
\varlimsup_{\zeta\downarrow 0} \varlimsup_{\tau\downarrow 0}
\varlimsup_{n\to\infty}\tfrac{1}{n}\log \bb P_{\subm}\Big[\pi^n\in (A_{k,l}^{\zeta,\tau})^{\complement}\Big]
\;\leq\; -l \,.
\end{equation}
This is a consequence of
Corollary \ref{cor_super_energy} and the fact that $(\pi^n*\iog)*\ioz -\pi^n*\ioz	=O (\pfrac{\tau}{\zeta})$, see Proposition~5.9 in \cite{FN2017} for details.

Another  technical problem that arises in this setting is the fact that the empirical measure does not have a density with respect to the Lebesgue measure. An extra family of sets is then defined to circumvent this issue.
Fix a sequence $\{ F_i\}_{i\geq 1}$ of smooth non negative functions dense, with respect to the uniform topology, in the subset of non-negative 
continuous functions.
 For $m\geq 1$ and $j\geq 1$,  define the set
\begin{equation}\label{E_set} E_m^j\; =\; \Big\{\pi\in \DM\;;\;0\leq \<\pi_t,F_i\>\leq \int_{0}^1 F_i(u)\,du
+\pfrac{1}{j}\Vert F'_i\Vert_\infty ,\,
0\leq t\leq T,\; i=1, \dots,m\,\Big\}\,.
\end{equation}
It is a simple task to check that $\DMO=\cap_{j\geq 1}\cap_{m\geq 1} E_m^j$.
Given $m\geq 1$ and $j\geq 1$,  the following limsup holds:
\begin{equation}\label{E_set_lim}
\varlimsup_{n\to\infty}\pfrac{1}{n}\log \bb P_{\subm}\big[\pi^n\in (E_m^j)^{\complement}\big]\;=\;-\infty\,.
\end{equation}
This result is very similar to the one in \cite[Subsection 6.3]{flm} and   \cite{FN2017}, thus its   proofs  is omitted here.
For the case $\theta\in(1,+\infty)$, we also need to assure that trajectories  that do not  conserve mass are negligible. We thus  introduce one more one set. For $\lambda>0$, 
let 
\begin{equation}\label{F_set}
\mc F_{\lambda}^\theta\;=\begin{cases}\;\big\{\pi\in\DM:|\<\pi_t,1\>-\<\pi_0,1\>|\leq \lambda,\;0\leq t\leq T\big\}\,,\; &\mbox{ if } \theta\in(1,+\infty)\,,\\\;\DM\,,\; &\mbox{ if } \theta\in(0,1)\,.
\end{cases}
\end{equation}
This is a closed set and below we  prove that it is super-exponentially small.
\begin{lemma}\label{lemma_5.1} For all $\theta\in(1,+\infty)$ and all $\lambda>0$, it holds that
\begin{equation}\label{F_set_lim}
\varlimsup_{n\to\infty}\tfrac{1}{n}\log \bb P_{\subm}\Big[ \pi^n\in(\mc F_{\lambda}^\theta)^{\complement}\Big]\;=\;-\infty\,.
\end{equation}
\end{lemma}
\begin{proof}
Let us appeal to the Harris graphical construction of the process.
 Let $N_t^{+,1}$ and $N_t^{-,1}$  be the Poisson processes associated to the site $x=1$, whose parameters are $\alpha n^{2-\theta}$ and $(1-\alpha) n^{2-\theta}$, respectively.  At an arrival of the Poisson process  $N_t^{+,1}$, if there is no  particle at the site $1$, a new particle is dropped there. And at an  arrival of the Poisson process  $N_t^{-,1}$, if there is a  particle at the site $1$, it leaves the system.
Analogously, let  $N_t^{+,n-1}$ and $N_t^{-,n-1}$  be the Poisson processes associated to the right site $x=n-1$, whose parameters are $\beta n^{2-\theta}$ and $(1-\beta) n^{2-\theta}$, respectively,  with the same action of creation and destruction of particles at the site $x=n-1$.
Since each particle contributes with a mass $1/n$ to the empirical measure, we get that
\begin{align*}
\Big[ \pi^n\in(\mc F_{\lambda}^\theta)^{\complement}\Big] & \;\subset\; 
\bigcup_{i\in\{+,-\}}\bigcup_{j\in\{1,n-1\}} \Big[\exists\, t\in[0,T]:  N_t^{i,j}\geq \lambda n\Big]\\
& \;\subset\; 
\bigcup_{i\in\{+,-\}}\bigcup_{j\in\{1,n-1\}} \Big[N_T^{i,j}\geq \lambda n\Big]\,.
\end{align*} 
By \eqref{sum_log_super}, in order to prove \eqref{F_set_lim}, 
it is enough to prove that 
\begin{equation*}
\varlimsup_{n\to\infty}\tfrac{1}{n}\log \bb P_{\subm}\Big[N_T^{i,j}\geq \lambda n\Big]\;=\;-\infty
\end{equation*}
for $i\in\{+,-\}$ and $j\in\{1,n-1\}$. Let us first review   standard facts on large deviations of i.i.d.\ random variables. Given i.i.d.\ random variables $\{X_i\}_{i\geq 1}$ with  Poisson distribution of parameter $a>0$, it is  deduced by Markov's inequality that
\begin{equation}\label{largePoisson}
\frac{1}{n}\log \bb P_{\subm}\Big[\frac{X_1+\cdots+X_n}{n}\geq x\Big]\;\leq\; x-a-x\log \frac{x}{a}\,,\quad \forall\, x>0\,.
\end{equation}
 Since the number of arrivals of $N_T^{i,j}$ is a Poisson process of parameter $cn^{2-\theta}$  for some $c>0$, and a sum of independent variables of Poisson distribution  has Poisson distribution whose parameter is given by the sum of the parameters, we have that 
\begin{equation*}
N_T^{i,j}\;\sim\; Y_1+\cdots+Y_{n}\,,
\end{equation*} 
where $Y_i\sim$ Poisson$(cn^{1-\theta})$. From \eqref{largePoisson}, 
\begin{align*}
\frac{1}{n}\log\bb P_{\subm} \Big[N_T^{i,j}\geq \lambda n\Big]&\;=\; \frac{1}{n}\log\bb P_{\subm} \Big[\frac{Y_1+\cdots+Y_{n}}{n}\geq \lambda \Big]\;\leq\; \lambda-cn^{1-\theta}-\lambda \log \bigg(\frac{\lambda}{cn^{1-\theta}}\bigg)\,,
\end{align*}
which converges to $-\infty$, when $n\to\infty$, since $\theta\in(1,+\infty)$, hence finishing the proof.
\end{proof}
\begin{lemma}\label{lem:current_exp}
For all $\theta\in(1,+\infty)$ and all $\lambda>0$, it holds
\begin{equation}\label{J_super_small}
\varlimsup_{n\to\infty}\tfrac{1}{n}\log \bb P_{\subm}\Big[ \tfrac{1}{n}J^n_{0,1}(t) >\lambda\Big]\;=\;\varlimsup_{n\to\infty}\tfrac{1}{n}\log \bb P_{\subm}\Big[ \tfrac{1}{n}J^n_{n-1,n}(t) >\lambda\Big]\;=\;-\infty\,.
\end{equation}
\end{lemma}
\begin{proof}
The argument is similar to the one in the  proof of the Lemma~\ref{lemma_5.1}. The current $J^n_{0,1}(t)$ of particles through the left boundary is stochastically dominated by a Poisson random variable $N_T^{i,j}$ of parameter $cn^{2-\theta}$  for some $c>0$. Thus, by the same large deviation argument of \eqref{largePoisson}, we get 
\begin{align*}
\frac{1}{n}\log\bb P_{\subm} \Big[\tfrac{1}{n}J^n_{0,1}(t)\geq \lambda \Big]&\;\leq \; \frac{1}{n}\log\bb P_{\subm} \Big[\tfrac{1}{n}N_T^{i,j}\geq \lambda \Big]\\
&\;=\;\frac{1}{n}\log\bb P_{\subm} \Big[N_T^{i,j}\geq \lambda n \Big] \;\leq\; \lambda-cn^{1-\theta}-\lambda \log \frac{\lambda}{cn^{1-\theta}}\,,
\end{align*}
leading to \eqref{J_super_small}.  The argument for $J^n_{n-1,n}(t)$ is analogous.
\end{proof}

To conclude this subsection, define the set 
\begin{equation}\label{G_set}
{\mc{G}}_{H,\zeta,\tau,\lambda,\delta,\eps}^{\theta,n,k,l,m,j}\;:=\;\{\pi^n \in  A_{k,l}^{\zeta,\tau}\cap E_m^j\cap \mc F_\lambda^\theta\}\cap B_{\delta,\eps}^H\;\subset\;\Ddiscreto\,,
\end{equation} 
where the sets $ A_{k,l}^{\zeta,\tau}$, $E_m^j$,  $\mc F_\lambda^\theta$ and $B_{\delta,\eps}^H$ were defined in \eqref{A_set}, \eqref{E_set}, \eqref{F_set} and \eqref{B_set}, respectively.  Since
\begin{equation*}
\begin{split}&
	\varlimsup_{n\to\infty}	\tfrac{1}{n}\log \bb P_{\subm}\Bigg[\,\Big({\mc G}_{H,\zeta,\tau,\lambda,\delta,\eps}^{\theta,n,k,l,m,j}\Big)^\complement\Bigg]\\&\leq  \max\Bigg\{
 \varlimsup_{n\to\infty}\pfrac{1}{n}
	\log \bb P_{\subm}\Big[\pi^n\in (A_{k,l}^{\zeta,\tau})^{\complement}\Big]\,,\;\varlimsup_{n\to\infty}\pfrac{1}{n}
	\log \bb P_{\subm}\Big[\pi^n\in (E_m^j)^{\complement}\Big]\,,\\
	& \quad\quad\quad\quad\varlimsup_{n\to\infty}\pfrac{1}{n}
	\log \bb P_{\subm}\Big[\pi^n\in (\mc F_\lambda^\theta)^{\complement}\Big]\,,\;\;\varlimsup_{n\to\infty}\,\pfrac{1}{n}
	\log \bb P_{\subm}\Big[(B^H_{\delta,\eps})^{\complement}\Big]\Bigg\}
	\end{split}
	\end{equation*}
	and  due to   \eqref{B_set_lim},  \eqref{A_set_lim},   \eqref{E_set_lim} and  \eqref{F_set_lim},
we deduce that
\begin{equation}\label{G_set_lim}
	\varlimsup_{\eps\downarrow 0} 	\varlimsup_{\zeta\downarrow 0} \varlimsup_{\tau\downarrow 0}\varlimsup_{n\to\infty}	\tfrac{1}{n}\log \bb P_{\subm}\Bigg[\,\Big({\mc{G}}_{H,\zeta,\tau,\lambda,\delta,\eps}^{\theta,n,k,l,m,j}\Big)^\complement\Bigg]\leq \,-\,l\,.
\end{equation}

\subsection{Radon-Nikodym derivative (continuation)}\label{R-N_der_cont}

In order to write the Radon-Nikodym derivative in a proper way we  start by introducing some notations.
Having in mind that  $\mc E((\pi*\iog)*\ioe)<\infty$ for all $\pi\in \DM$, we define
 the functional
\begin{equation}\label{Jmuito}
J_{H,\zeta,\tau,\lambda,\eps}^{\theta,k,l,m,j}(\pi)\;=\;
\begin{cases}
\;\ell_H^\theta\big((\pi*\iog)*\ioe\big)-\Phi_H\big((\pi*\iog)*\ioe\big), &  \mbox{if}\,\,\,\,\pi\in A_{k,l}^{\zeta,\tau}\cap E_m^j \cap \mc F_\lambda^\theta\,,\\ 
\;+\infty, &\mbox{otherwise}\,.
\end{cases}
\end{equation}
The next  result establishes the connection between  $J_{H,\zeta,\tau,\lambda,\eps}^{\theta,k,l,m,j}(\pi)$ and  the functional $J_H^\theta(\pi)$ defined in \eqref{functional:J}.
\begin{proposition}\label{labels} For all $\pi\in \DM$,
	\begin{equation*}
	\varlimsup_{\eps\downarrow 0}
	\varlimsup_{l\to\infty}\varlimsup_{k\to\infty}
	\varlimsup_{\zeta\downarrow 0}
	\varlimsup_{\tau\downarrow 0}
	\varlimsup_{\lambda\downarrow 0}
	\varlimsup_{j\to\infty}\varlimsup_{m\to\infty}	
	J_{H,\zeta,\tau,\lambda,\eps}^{\theta,k,l,m,j}(\pi)
	\;\geq\; J_H^\theta(\pi)\,.
	\end{equation*}
\end{proposition}

\begin{proof} The proof of this proposition is very similar to the proof of Proposition 5.12 of \cite{FN2017}, except by the presence of  an extra limsup as $\lambda\downarrow 0$.  For   $\pi\in \DM$, if $\pi\notin\DMO$ then there exist  $m$ and $j$ such that $\pi\notin E_m^j$. Therefore,
	\begin{equation*}
	\varlimsup_{j\to\infty}\varlimsup_{m\to\infty}
	J_{H,\zeta,\tau,\lambda,\eps}^{\theta,k,l,m,j}(\pi)\;=\;
	\begin{cases}
	\;\ell_H^\theta\big((\pi*\iog)*\ioe\big)-\Phi_H\big((\pi*\iog)*\ioe\big), &  
	\mbox{if}\,\,\,\,\pi\in A_{k,l}^{\zeta,\tau}\cap \DMO \cap \mc F_\lambda^\theta,\\ 
	\;+\infty, &\mbox{otherwise}.
	\end{cases}
	\end{equation*}
	Recall in \eqref{F_set}  the definition of $\mc F^\theta_\lambda$ and recall in \eqref{Ftheta} the definition of  $\mc F^\theta$.
Taking the limsup as $\lambda\downarrow 0$ we obtain that
	\begin{align*}
	& \varlimsup_{\lambda\downarrow 0}	\varlimsup_{j\to\infty}\varlimsup_{m\to\infty}
	J_{H,\zeta,\tau,\lambda,\eps}^{\theta,k,l,m,j}(\pi)\;\geq\;
	\begin{cases}
	\;\ell_H^\theta\big((\pi*\iog)*\ioe\big)-\Phi_H\big((\pi*\iog)*\ioe\big), &  
	\mbox{if}\,\,\,\,\pi\in A_{k,l}^{\zeta,\tau}\cap \DMO \cap \mc F^\theta,\\ 
	\;+\infty, &\mbox{otherwise}.
	\end{cases}
	\end{align*}
Recall \eqref{A_set}. Taking  the limsup as $\tau\downarrow 0$ and	then as $\zeta\downarrow 0$,
	\begin{equation*}\begin{split}&
	\varlimsup_{\zeta\downarrow 0} 	\varlimsup_{\tau\downarrow 0}
	\varlimsup_{\lambda\downarrow 0}	\varlimsup_{j\to\infty}\varlimsup_{m\to\infty}
	J_{H,\zeta,\tau,\lambda,\eps}^{\theta,k,l,m,j}(\pi)\;	\geq \;
	\begin{cases}
	\;\ell_H^\theta\big(\pi*\ioe\big)-\Phi_H\big(\pi*\ioe\big),  & \!\!\! 
	\mbox{if }\pi\in A_{k,l+2}\cap \DMO\cap \mc F^\theta\,,\\ 
	\;+\infty, &\mbox{otherwise}\,,
	\end{cases}
	\end{split}
	\end{equation*}
	see \cite{FN2017} for details on this step. Since $\{\pi:\mc E(\pi)\leq l+2\}\subset\DMO$, taking  now the limsup as $k\to\infty$ we  obtain that
	\begin{equation*}
	\begin{split}
	&\varlimsup_{k\to \infty}
	\varlimsup_{\zeta\downarrow 0} 	\varlimsup_{\tau\downarrow 0}
	\varlimsup_{\lambda\downarrow 0}	\varlimsup_{j\to\infty}\varlimsup_{m\to\infty}
	J_{H,\zeta,\tau,\lambda,\eps}^{\theta,k,l,m,j}(\pi)\;\geq\; 
	\begin{cases}
	\;\ell_H^\theta\big(\pi*\ioe\big)-\Phi_H\big(\pi*\ioe\big), &  
	\mbox{if}\,\, \pi\in \mc F^\theta \text{ and }\mc E(\pi)\leq l+2,\\ 
	\;+\infty\,, &\mbox{otherwise.}
	\end{cases}
	\end{split}
	\end{equation*}
	Taking now the limsup as $l\to\infty$, we get
	\begin{equation*}
	\begin{split}
	& \varlimsup_{l\to \infty}
	\varlimsup_{k\to \infty}
	\varlimsup_{\zeta\downarrow 0} 	\varlimsup_{\tau\downarrow 0}
	\varlimsup_{\lambda\downarrow 0}	\varlimsup_{j\to\infty}\varlimsup_{m\to\infty}
	J_{H,\zeta,\tau,\lambda,\eps}^{\theta,k,l,m,j}(\pi)\; \geq \;
	\begin{cases}
	\;\ell_H^\theta\big(\pi*\ioe\big)-\Phi_H\big(\pi*\ioe\big), &  
	\mbox{if}\,\, \pi\in \mc F^\theta \text{ and }\mc E(\pi)<\infty,\\ 
	\;+\infty\,, &\mbox{otherwise.}
	\end{cases}
	\end{split}
		\end{equation*}
	For $\pi$ such that $\mc E(\pi)<\infty$ it holds that  $\pi_t(du)=\rho_t(u)du$,
	where $\rho$ has 
	well-defined  limits at the boundary. Thus,  taking the limsup as $\eps\downarrow 0$, we obtain
	\begin{equation*}
	\varlimsup_{\eps\downarrow 0} 
	\varlimsup_{l\to \infty}
	\varlimsup_{k\to \infty}
	\varlimsup_{\zeta\downarrow 0} 	\varlimsup_{\tau\downarrow 0}
	\varlimsup_{\lambda\downarrow 0}	\varlimsup_{j\to\infty}\varlimsup_{m\to\infty}
	J_{H,\zeta,\tau,\lambda,\eps}^{\theta,k,l,m,j}(\pi)\; \geq\; J_H^\theta(\pi)\;,
	\end{equation*}
	concluding the proof.
\end{proof}

One ingredient in the proof of large deviations is to  restrict the Radon-Nikodym derivative given  in \eqref{correct_RN_1} to the set ${\mc{G}}_{H,\zeta,\tau,\lambda,\delta,\eps}^{\theta,n,k,l,m,j}$ defined in \eqref{G_set}, which encodes all the   sets introduced in  Subsection~\ref{small_sets} and then to show that this ``restricted Radon-Nikodym derivative'' is close to 
 an exponential of minus $n$ times a functional of the empirical measure, that is, we must assure that
\begin{equation}\label{radon J indices}
\begin{split}
&\frac{\dradon\bb P_{\subm}}{\dradon\bb P^{H}_{\subm}}\Bigg|_{\mc F_T}\,\cdot\,\;\textbf{1}_{{\mc{G}}_{H,\zeta,\tau,\lambda,\delta,\eps}^{\theta,n,k,l,m,j}}\;=\;\textbf{1}_{{\mc{G}}_{H,\zeta,\tau,\lambda,\delta,\eps}^{\theta,n,k,l,m,j}}\,\cdot\,\exp{\bigg\{ -n\Big[J_{H,\zeta,\tau,\lambda,\eps}^{\theta,k,l,m,j}(\pi^n)
	+ \textrm{err}_{H}^{\theta}(n,\tau,\eps,\delta)\Big]\bigg\}}\,,
\end{split}
\end{equation}
with
\begin{equation}\label{error}
\varlimsup_{\delta\downarrow 0}\varlimsup_{\eps\downarrow 0}\varlimsup_{\tau\downarrow 0}\varlimsup_{n\to\infty}\big|\textrm{err}_{H}^{\theta}(n,\tau,\eps,\delta)\,\big|\;=\;0
\end{equation}
for all $\theta>0$,  $H\in {\bf{C}}_\theta$, where the dependence on $T$ has been omitted.
Here we do not present the derivation of \eqref{radon J indices} because it is very similar to what is done in \cite[Subsection~5.1]{FN2017}. We only advertise that the order of the limits above
can not be changed. For example, one term of $\textrm{err}_{H}^{\theta}(n,\tau,\eps,\delta)$ is of order $\frac{\tau}{\eps}$.
The expression \eqref{radon J indices} is the appropriate form  for the Radon-Nikodym derivative to be used in the next subsection.
Although the relationship between $J_{H,\zeta,\tau,\lambda,\eps}^{\theta,k,l,m,j}(\pi)$ and $J_H^\theta(\pi)$ was presented in  Proposition \ref{labels}, we will use $J_{H,\zeta,\tau,\lambda,\eps}^{\theta,k,l,m,j}(\pi)$ instead of $J_H^\theta(\pi)$ to  allow  the application of Minimax Lemma.

\subsection{Upper bound for compact sets}\label{sub_5.3}
To reach the upper bound for compact sets we have to recall  the Minimax Lemma, see \cite[page 373, Lemma 3.3]{kl}. We start with the upper bound for open sets.
Let $\mc O\subseteq\DM$ be an open set and fix a function  $H\in {\bf{C}}_\theta$. By a similar computation  presented in the begin of Section \ref{s5}, we have, for all $\theta>0$,  $H\in {\bf{C}}_\theta$ , $\lambda>0$, $\delta>0$, $k,l,m,j\in \bb N$, $\zeta,\tau,\eps>0$,
\begin{align}
&\varlimsup_{n\to\infty}\tfrac{1}{n}\log \bb Q_{\subm}[\mc O] 
= \varlimsup_{n\to\infty}\tfrac{1}{n}\log \bb P_{\subm}[\pi^n\in\mc O]\notag\\
& \leq  \max\Bigg\{\varlimsup_{n\to\infty}\pfrac{1}{n}
\log \bb P_{\subm}\Big[\{\pi^n\in\mc O\}\cap {\mc{G}}_{H,\zeta,\tau,\lambda,\delta,\eps}^{\theta,n,k,l,m,j}\Big]\,,\;\, \varlimsup_{n\to\infty}\pfrac{1}{n}
\log \bb P_{\subm}\bigg[\Big({\mc{G}}_{H,\zeta,\tau,\lambda,\delta,\eps}^{\theta,n,k,l,m,j}\Big)^\complement\bigg]\Bigg\}\,,\label{prob1}
\end{align}
where 
\begin{equation}\label{R_lim}
\varlimsup_{l\to\infty}\varlimsup_{\eps\downarrow 0}\varlimsup_{\tau\downarrow 0}\varlimsup_{\zeta\downarrow 0}\varlimsup_{n\to\infty}\pfrac{1}{n}
\log \bb P_{\subm}\bigg[\Big({\mc{G}}_{H,\zeta,\tau,\lambda,\delta,\eps}^{\theta,n,k,l,m,j}\Big)^\complement\bigg]\;=\;-\infty\,,
\end{equation}
due to   \eqref{G_set_lim}.
Now, we use  the expression \eqref{radon J indices} of
 the Radon-Nikodym derivative to estimate the first probability on \eqref{prob1}, that is:
\begin{equation*}
\begin{split}
&\bb P_{\subm}\Big[\{\pi^n\in\mc O\}\cap {\mc{G}}_{H,\zeta,\tau,\lambda,\delta,\eps}^{\theta,n,k,l,m,j}\Big]\;=\;\bb E^{H}_{\subm}\Bigg[\,\frac{\dradon\bb P_{\subm}}{\dradon\bb P^{H}_{\subm}}\Bigg|_{\mc F_T}\,\cdot\,\;\textbf{1}_{{\mc{G}}_{H,\zeta,\tau,\lambda,\delta,\eps}^{\theta,n,k,l,m,j}}
\,\cdot\,\textbf 1_{\{\pi^n\in \mc O\}}\,\Bigg]\\
&=\;\bb E^{H}_{\subm}\Bigg[\,
 \textbf{1}_{{\mc{G}}_{H,\zeta,\tau,\lambda,\delta,\eps}^{\theta,n,k,l,m,j}}\,\cdot\,\exp\Big\{ -n\big[J_{H,\zeta,\tau,\lambda,\eps}^{\theta,k,l,m,j}(\pi^n)
+ \textrm{err}_{H}^{\theta}(n,\tau,\eps,\delta)\big]\Big\}
\,\cdot\,\textbf 1_{\{\pi^n\in \mc O\}}\,\Bigg]\,.\\
\end{split}
\end{equation*}
Therefore,
\begin{equation*}
\pfrac{1}{n}\log \bb P_{\subm}\Big[\{\pi^n\in\mc O\}\cap {\mc{G}}_{H,\zeta,\tau,\lambda,\delta,\eps}^{\theta,n,k,l,m,j}\Big]\;\leq\;  \sup_{\pi\in\mc O}\Big\{-J_{H,\zeta,\tau,\lambda,\eps}^{\theta,k,l,m,j}(\pi)
-\textrm{err}_{H}^{\theta}(n,\tau,\eps,\delta)\Big\}\,.
\end{equation*}
Optimizing over all the parameters  $\tau,\eps,\zeta,\delta,\lambda, k,l,m,j,H$, it yields
\begin{equation}\label{infsup}
\begin{split}
&\varlimsup_{n\to\infty}\pfrac{1}{n}\log \bb Q_{\subm}[\,\mc O\,]\leq\!\!\!\! \inf_{\at{\tau,\eps,\zeta,\delta,\lambda,}{k,l,m,j,H}}\sup_{\pi\in\mc O}\max\Big\{ 
-J_{H,\zeta,\tau,\lambda,\eps}^{\theta,k,l,m,j}(\pi)
-\textrm{err}_{H}^{\theta}(n,\tau,\eps,\delta)\,,\;\mc R_{H,\lambda,\delta}^{\theta,k,l,m,j}(\zeta,\tau,\eps)
\Big\}\,.
\end{split}
\end{equation}
To interchange the supremum and the infimum above, we start by observing that  for fixed parameters $\tau,\eps,\zeta,\delta,\lambda,k,l,m,j,H$, the functional 
	\begin{equation*}
	\pi\mapsto \max\Big\{ 
	-J_{H,\zeta,\tau,\lambda,\eps}^{\theta,k,l,m,j}(\pi)
	-\textrm{err}_{H}^{\theta}(n,\tau,\eps,\delta)\,,\;\mc R_{H,\lambda,\delta}^{\theta,k,l,m,j}(\zeta,\tau,\eps)
	\Big\}
	\end{equation*}
	is upper semi-continuous in $\DM$. The proof of this result is similar to the proof of Proposition~5.11 in  \cite{FN2017}.
Thus, we can apply the  Minimax Lemma, see \cite[page 373, Lemma 3.3]{kl}, hence interchanging the supremum with the infimum in \eqref{infsup}, and passing the bound to compacts sets.
Then, for all $\mc K\subset \DM$ compact, 
\begin{equation*}
\begin{split}
&\varlimsup_{n\to\infty}\pfrac{1}{n}\log \bb Q_{\subm}[\,\mc K\,]\leq\sup_{\pi\in\mc K}
 \inf_{\at{\tau,\eps,\zeta,\delta,\lambda,}{k,l,m,j,H}}\max\Big\{ 
\big[-J_{H,\zeta,\tau,\lambda,\eps}^{\theta,k,l,m,j}(\pi)
-\textrm{err}_{H}^{\theta}(n,\tau,\eps,\delta)\big]\,,\;\mc R_{H,\lambda,\delta}^{\theta,k,l,m,j}(\zeta,\tau,\eps)
\Big\}\,.
\end{split}
\end{equation*}
Putting together Proposition \ref{labels},  \eqref{R_lim} and  \eqref{error}, we  deduce:
\begin{proposition}[Upper bound for compact sets] 
	For every $\mc K$ compact subset of $\DM$,
	\begin{equation*}
	\varlimsup_{n\to\infty}\pfrac{1}{n}\log \bb Q_{\subm}[\mc K]
	\;\leq\; -\inf_{\pi\in\mc K} \I(\pi|\gamma)\,.
	\end{equation*}
\end{proposition}

\subsection{Upper bound for closed sets}\label{sub_5.4} 
In  Subsection~\ref{sub_5.3}, we already have the large deviations upper\break bound for closed sets. The  extension to closed sets is a standard routine based on  \textit{exponential tightness}. The exponential tightness is defined as the  existence of compact sets $K_\ell\subset \DM$ such that
\begin{equation}\label{exponential_tightness}
 \limsup_{n\to\infty}\pfrac{1}{n}\log  \bb Q_{\subm}\big[K_\ell^\complement\big]\;\leq\; -\ell\;,\quad\quad\forall\; \ell\in \bb N\;.
\end{equation}
Let $\mc C\subset \DM$ be a closed set. Assuming exponential tightness,  we have that
\begin{align*}
\limsup_{N\to\infty}\pfrac{1}{n}\log  \bb Q_{\subm}\big[\mc C\big]& \;\leq \; \limsup_{n\to\infty}\pfrac{1}{n}\log  \bb Q_{\subm}\big[\mc C\cap K_\ell\big]+
\limsup_{n\to\infty}\pfrac{1}{n}\log  \bb Q_{\subm}\big[ K_\ell^\complement\big]\\
& \;\leq \; \limsup_{n\to\infty}\pfrac{1}{n}\log  \bb Q_{\subm}\big[\mc C\cap K_\ell\big]-\ell\,.
\end{align*}
Hence, since the set $\mc C\cap K_\ell$ is compact and $\ell$ is arbitrary,  the upper bound for closed sets will follow from the upper bound for compact sets. 

The proof of the exponential tightness \eqref{exponential_tightness} is somewhat technical and  follows the same steps of \cite[Section 5.3]{FN2017}\footnote{In its hand, \cite[Section 5.3]{FN2017} is essentially a detailed version of \cite[pp. 271--273]{kl}.}. For this reason, we discuss  only what  considered needs to be checked for our model.
 With respect to \cite[Section 5.3]{FN2017}, the only and somewhat crucial point to be adapted is to find  a positive mean one martingale with respect to the natural filtration,
\begin{equation*}
\begin{split}
 & M^{a,H}_t\;:=\;\exp{\Big\{an\Big[\<\pi^n_t,H\>-\<\pi^n_0,H\>}
- \int_0^t U_n^a(H,s,\eta_s)\, ds\Big]\Big\}\,,
\end{split}
\end{equation*}
 where $|U_n^a(H,s,\eta_s)|$ is  uniformly bounded in $n\in \bb N$. 
  This claim is a consequence of the general fact that the Radon-Nikodym derivative between two Markov processes is a a positive mean one martingale with respect to the natural filtration together with formula \eqref{correct_RN_1} choosing $aH$ in lieu of $H$. 
In resume, we have therefore achieved:
\begin{proposition}[Upper bound for closed sets] 
For every $\mc C$ closed subset of $\DM$,
 \begin{equation*}
 \limsup_{n\to\infty}\tfrac{1}{n}\log \bb Q_{\subm}\big[\mc C\big]
\;\leq\; -\inf_{\pi\in\mc C}\I(\pi|\gamma)\,.
 \end{equation*}
\end{proposition}

\section{Large deviations lower bound}\label{sec6}

The proof of the lower bound  in the case $\theta\in (0,1)$ is quite similar to  \cite{blm} or \cite{flm}   (in dimension $d=1$), which  correspond to $\theta=0$ in our setting. 
We henceforth study in detail  the
 case $\theta\in(1,+\infty)$ following the more recent approach of \cite{LandimTsunoda}. 
  Due to the presence of large deviations from the initial measure we are not allowed to apply \cite[Theorem 2.4]{JLLV} and an $\I$-density argument is required here as in the framework of \cite{LandimTsunoda}.

\subsection{Lower bound for smooth profiles}
 The next two propositions are immediate consequences of the definition of  $J_G^\theta$ and show that solutions of the perturbed partial differential equations \eqref{eq:edpasy_menor_um} or \eqref{eq:edpasy=1_maior_um} depending on whether  $\theta\in(0,1)$ or $\theta\in(1,+\infty)$ lead to a simpler representation of the rate function.
\begin{proposition}\label{prop6364} Consider  $\theta\in(0,1)$ or $\theta\in(1,+\infty)$ and recall the definition of $\C$.   Given  $H\in\C$ , let  $\rho^H$ be the unique  weak solution of \eqref{eq:edpasy_menor_um} if $\theta\in(0,1)$ or the unique  weak solution of  \eqref{eq:edpasy=1_maior_um} if $\theta\in(1,+\infty)$. Then
 \begin{align*}
 \sup_{G\in \C}  J_G^\theta(\rho^H) & \;=\;\sup_{G\in \C} \Big\{\ell_G^\theta(\rho^H)-\Phi_G(\rho^H)\Big\} \\
 & \;=\; \sup_{G\in \C} \Bigg\{2\int_0^t\<\chi(\rho^H_s) \,\p_uH_s, \p_uG_s\>\,ds - \int_0^T \< \chi(\rho_s^H), (\p_u H_s)^2 \>\,ds \Bigg\}\\
 &  \;=\;  \int_0^T \< \chi(\rho_s^H), (\p_u H_s)^2 \>\,ds\,.
\end{align*}
\end{proposition} 
Proposition \ref{prop6364}  motivates the next definition.
\begin{definition}\label{Pi}
Denote by $\Pi$ the subspace of $\DMO$ consisting of all paths $\pi_t(du)=\rho_t(u)\,du$ for which there exists some 
$H\in \C$ such that  $\rho=\rho^H$ is the unique weak solution of \eqref{eq:edpasy_menor_um} if $\theta\in(0,1)$ or the unique weak solution of \eqref{eq:edpasy=1_maior_um} if $\theta\in(1,+\infty)$. 
\end{definition}
The next two propositions provide conditions to assure that a  profile $\rho$ is solution of the  corresponding hydrodynamic equation (according to each regime of $\theta$) for some $H$. That is, conditions to assure that $\rho\in \Pi$. Proposition \ref{elliptic_<1} is well known in the literature and it is included here for  sake of completeness.
\begin{proposition}\label{elliptic_<1}
Let $\theta\in(0,1)$.
Let  $\rho\in C^{1,2}$  such that
$0<\eps\leq \rho \leq 1-\eps$
for some $\eps>0$.  Then, there exists an unique  (strong) solution $H$ 
of the elliptic equation 
\begin{numcases}{}
 \p_u^2 H_t(u)\,+ \,\frac{\p_u\big(\chi(\rho_t(u))\big)}{\chi(\rho_t(u))}\,\p_u H_t(u)\, = \,\frac{\Delta \rho_t(u)\,-\, \partial_t \rho_t(u) }{2\,\chi(\rho_t(u))} \,,\, \forall u\in(0,1) &   \label{eq61}\\
  H_t (0)=0 & \label{eq62}\\
 H_t (1)=0 & \label{eq63}
\end{numcases}	
\end{proposition}

\begin{proof}  
Fix  $t\in[0,T]$.
Since \eqref{eq61} is a linear  ODE of second order on $H$, we  solve it, getting
\begin{equation}\label{eq_geral}
\begin{split}
H_t(u)\;=\;& H_t(0)+
\big(2\chi(\rho_t(0))\p_uH_t(0)-\p_u\rho_t(0)\big)\int_0^u\frac{1}{2\chi(\rho_t(v))}\,dv\\
&+
\int_0^u\frac{\p_u\rho_t(v)-\p_t\int_0^v\rho_t(w)\,dw}{2\chi(\rho_t(v))}\,dv\,.
\end{split}
\end{equation}
Taking $u=1$ and then applying the boundary conditions \eqref{eq62} and \eqref{eq63} in the equality \eqref{eq_geral} above, we get
\begin{equation}\label{IIIa}
\p_u H_t (0)\;=\;\frac{1}{2\chi(\rho_t(0))}\Big\{\p_u\rho_t(0)-\frac{\bb I_t}{I_t}\Big\}\,,
\end{equation}
where \begin{equation}\label{III}
I_t:=	\int_0^1\frac{1}{2\chi(\rho_t(v))}\,dv\quad\mbox{ and }\quad\bb I_t:=	\int_0^1\frac{\p_u\rho_t(v)-\p_t\int_0^v\rho_t(w)\,dw}{2\chi(\rho_t(v))}\,dv\,.
\end{equation}
In other words, \eqref{IIIa} is the right guess for $\p_u H_t(0)$ in order to achieve the solution of the elliptic PDE in the statement of the proposition. Coming back to \eqref{eq_geral}, we then apply \eqref{eq62} and  \eqref{IIIa}, which leads us to 
\begin{align*}
H_t(u) & \;=\;
\int_0^u\frac{\frac{-\bb I_t}{I_t}+\p_u\rho_t(v)-\p_t\int_0^v\rho_t(w)\,dw}{2\chi(\rho_t(v))}\,dv\\
& \;=\;-\frac{\bb I_t}{I_t}\int_0^u\frac{1}{2\chi(\rho_t(v))}\,dv+
\int_0^u\frac{\p_u\rho_t(v)-\p_t\int_0^v\rho_t(w)\,dw}{2\chi(\rho_t(v))}\,dv\,,
\end{align*}
and it is straightforward  to check that this is the required solution of the elliptic PDE. 
\end{proof}

\begin{proposition}\label{elliptic_>1}
Let $\theta\in(1,+\infty)$.
Consider  $\rho\in C^{1,2}$  such that
$0<\eps\leq \rho \leq 1-\eps$
for some $\eps>0$ and $\partial_t\int_0^1\rho_t(z)\,dz=0$.  Then, up to an additive constant, there exists an unique  (strong) solution $H$ 
of the elliptic equation 
\begin{numcases}{}
 \p_u^2 H_t(u)\,+ \,\pfrac{\p_u\big(\chi(\rho_t(u))\big)}{\chi(\rho_t(u))}\,\p_u H_t(u)\, = \,\pfrac{\Delta \rho_t(u)\,-\, \partial_t \rho_t(u) }{2\,\chi(\rho_t(u))} \,,\,  \forall u\in(0,1) & \label{eq6.15}\\
 \p_u H_t (0)=\pfrac{1 }{2\,\chi(\rho_t(0))}\p_u\rho_t(0)& \label{eq6.16}\\
 \p_u H_t (1)= \pfrac{1 }{2\,\chi(\rho_t(1))}\p_u\rho_t(1)& \label{eq6.17}
\end{numcases}
\end{proposition}

\begin{proof}  
Fix  $t\in[0,T]$. Solving the linear ODE of second order \eqref{eq6.15}, we get
\begin{equation*}
H_t(u):=H_t(0)+
\int_0^u\frac{2\chi(\rho_t(0))\p_uH_t(0)-\p_u\rho_t(0)+\p_u\rho_t(v)-\p_t\int_0^v\rho_t(w)\,dw}{2\chi(\rho_t(v))}\,dv\,.
\end{equation*}
The boundary condition \eqref{eq6.16} then leads us to
\begin{equation*}
H_t(u)\;=\;H_t(0)+
\int_0^u\frac{\p_u\rho_t(v)-\p_t\int_0^v\rho_t(w)\,dw}{2\chi(\rho_t(v))}\,dv\,.
\end{equation*}
Keeping in mind that $\partial_t\int_0^1\rho_t(z)\,dz=0$ it is straightforward to check that the expression on the right hand of the  above expression satisfies \eqref{eq6.17} regardless the chosen value of $H_t(0)$. 	
\end{proof}

\begin{proposition}\label{lower subset}
 Let $\mc O$ be an open set of $\DM$. Then
\begin{equation*}
 \varliminf_{N\to\infty}
\frac{1}{N}\log\bb Q_{\subm}\big[\,\mc O\,\big]\;\geq\; -\inf_{\pi\in \mc O\cap \Pi}\I(\pi|\gamma)\;.
\end{equation*}
\end{proposition}

The proof of the inequality above relies  on the hydrodynamic limit for the perturbed process and  Proposition~\ref{prop6364}. It follows  the same lines of  \cite[Chapter 10]{kl} or \cite{FN2017}. Let
\begin{equation}\label{ent}
\bs{H} \big(\bb P_{\subm}^{H}\vert\bb P_{\subm}\big)
\;:=\;\bb E_{\subm}^{H}\Big[\log\radonN\,\Big]
\;=\;-\bb E_{\subm}^H\Big[\log\radonNinv\,\Big]
\end{equation}
 be the so-called \textit{relative entropy} of $\bb P_{\subm}^{H}$ with respect to $\bb P_{\subm}$.

\begin{lemma}\label{entropy}
Let  $H\in \C$.
Then
\begin{equation*}
 \lim_{n\to\infty}\frac{1}{n}\bs{H} \big(\bb P_{\subm}^{H}\vert\bb P_{\subm}\big)\;=\; \I(\rho^H|\gamma)\,,
\end{equation*}
where  $\rho^H$ is the unique weak solution of \eqref{eq:edpasy_menor_um} if $\theta\in(0,1)$, or the unique weak solution of \eqref{eq:edpasy=1_maior_um} if $\theta\in(1,+\infty)$. 
\end{lemma}

\begin{proof} Recall the definition of $B_{\eps,\delta}^{H,\theta}$ in \eqref{B_set}, which is super-exponentially small,  see \eqref{B_set_lim}. On the $B_{\eps,\delta}^{H,\theta}$, we have that the Radon-Nikodym derivative $\radonN$ is equal to
\begin{equation}\label{radonJJ}
\exp{\Big\{ N\Big[J_H^\theta\Big((\pi^N*\iog)*\ioe\Big)
 + O_{H,T,\eps,\gamma}(\pfrac{1}{N})+O(\delta)+O_H(\eps)+O_H(\pfrac{\gamma}{\eps})\Big]\Big\}}\,.
\end{equation}
The proof of the above assertion  is  technical and follows the same steps of~\cite{FN2017}.
In view of \eqref{ent} for the relative entropy, 
\begin{equation}\label{ent1} 
\frac{1}{n}\bs{H} \big(\bb P_{\subm}^{H}\vert\bb P_{\subm}\big)\;=\;
\frac{1}{n}\bb E_{\subm}^{H}\Big[\log\radonN\,\textbf{1}_{B_{\eps,\delta}^{H,\theta}}\Big]+\frac{1}{n}\bb E_{\subm}^{H}\Big[\log\radonN\,\textbf{1}_{(B_{\eps,\delta}^{H,\theta})^\complement}\Big]\,,
\end{equation}
where the  $B_{\eps,\delta}^{H,\theta}$ has been defined in \eqref{G_set}.  By \eqref{G_set_lim}, the complement of this set is super-exponentially small with respect to $\bb P_{\nu_{\gamma(\cdot)}^n}$. We  affirm now that the complement  is super-exponentially small also with respect to $\bb P_{\subm}^{H}$.
Indeed, by \eqref{correct_RN_1} there exists a constant $C(H,T)>0$ such that  
\begin{equation*}
 \bb P_{\subm}^{H}\Big[(B_{\eps,\delta}^{H,\theta})^\complement\Big]\;=\;
\bb E_{\subm}\Big[\,\radonN \,\textbf{1}_{(B_{\eps,\delta}^{H,\theta})^\complement}\Big]
\;\leq\; e^{C(H,T)n}\bb P_{\subm}\Big[(B_{\eps,\delta}^{H,\theta})^\complement\Big]
\end{equation*}
and  by \eqref{B_set_lim} we get
\begin{equation*}
 \varlimsup_{\eps\downarrow 0} \varlimsup_{n\to\infty}\tfrac 1n \log \bb P_{\subm}^H\Big[\big (B_{\eps,\delta}^{H,\theta}\big)^\complement\Big]\;=\;-\infty
\end{equation*}
concluding the proof of the affirmation. By the previous limit and  since $\frac{1}{n}\log\radonN$ is bounded, the right hand
side of \eqref{ent1} can be written as 
\begin{equation}\label{entrop}
\frac{1}{n}\bb E_{\subm}^H\Big[\log\radonN\,\textbf{1}_{B_{\eps,\delta}^{H,\theta}}\Big]+o_n(1)\,.
\end{equation}
By Theorem~\ref{thm:hid_lim_weak}, under $\bb P_{\subm}^H$ the probability concentrates on $\rho^H$. Since the functional $J_H^\theta((\pi^N*\iog)*\ioe)$ is continuous in the Skohorod topology,  recalling \eqref{radonJJ} the proof ends.
\end{proof}

\subsection{The \texorpdfstring{$\I$}{I}-density}
In the previous subsection we have achieved the lower bound for smooth profiles. Our task now consists on extending it to any profile. We start with the definition of $\I$-density.

\begin{definition}
Let $A$ be a subset of $\DM$. The set $A$ is said to be $\I(\cdot |\gamma)$-dense if for any $\pi\in \DM$
such that $\I(\pi|\gamma) < \infty$ there exists a sequence $\{\pi_n : n \geq  1\}$ in $A$ such that 
\begin{equation*}
\pi_n\to \pi \text{  in } \DM \qquad\text{ and }\qquad \I(\pi_n|\gamma)\to\I(\pi|\gamma)\,.
\end{equation*}
\end{definition}
Recall Definition~\ref{Pi}. 
The main result to be proved now is:
\begin{theorem}\label{Idensity}
The set $\Pi$ is $\I$-dense.
\end{theorem} 
The statement above does not involve probability: it is a purely analytical result. Thus, since the $\I$ functional for $\theta\in(0,1)$ coincides with the rate functional of \cite{blm} under the assumption that the external field there considered is null, we thus may apply \cite[Theorem 5.1]{blm} in this case. 

From this point on we will deal only with the case $\theta\in(1,+\infty)$, where the proof of Theorem~\ref{Idensity} is split into intermediate lemmas. We start with a key technical result in the arguments, in whose proof we mix ideas from \cite{flm} and \cite{LandimTsunoda}. 
\begin{proposition}\label{prop68} Let $\theta\in(1,+\infty)$. There exists a constant $\tilde{C}_0>0$ such that, for any $\rho\in \DM$, it holds that
\begin{align}\label{obj}
\int_0^T\int_0^1 \frac{(\p_u\rho_t(u)\big)^2}{\chi(\rho_t(u))}dudt\;\leq \; \tilde{C}_0 \big(\I (\rho)+1\big)\,.
\end{align}
\end{proposition}
\begin{proof}
In what follows, assume $\pi\in\DM$   to be such that $\I(\pi|\gamma)<\infty$, otherwise  \eqref{obj} is trivial. Since $\I(\pi\gamma)<\infty$, then $\pi(t,du)=\rho(t,u)du$ with $\rho\in\Sob $ and from an integration by parts we have that 
\begin{align*}
\I(\pi|\gamma)  \;=\; \sup_{H\in \C }J^\theta_H(\rho) \;=\; \sup_{H\in \C} \Big\{L_H(\rho) + B_H(\rho) \Big\}\,,
\end{align*}
where 
\begin{align*}
L_H(\rho)& \;=\; \<\rho_T,H_T\> - \<\rho_0,H_0\>- \int_0^T\<\rho_s, \p_s H_s\>ds\qquad \text{ and }\\
B_H(\rho) & \;=\; \int_0^T\<\p_u\rho_s, \p_u H_s\>ds - \int_0^T\<\chi\big(\rho_s\big), \big(\p_u H_s\big)^2\>ds\,.
\end{align*} For $a\in(0,1)$, let $h_a:[0,1]\to \bb R$ be the  function defined by
\begin{align*}
h_a(x)\;=\;(x+a)\log(x+a) +(1-x+a)\log(1-x+a) 
\end{align*}
 whose first and second derivatives are, respectively,
\begin{align*}
 h_a'(x) & \;=\;  \log \bigg( \frac{x+a}{1-x+a}\bigg) \qquad \text { and } \qquad  h_a''(x) \;=\; \frac{1+2a}{(x+a)(1-x+ a)}\,.
\end{align*}
It is elementary to check that $-\log 2\leq h_a(x)\leq \log 4$ for all $x\in(0,1)$. Let
\begin{align*}
H_\rho \;:=\; h_a'(\rho)\,.
\end{align*}
 Since the space integrals above are with respect to the Lebesgue measure, we can see the integrated functions as functions defined on the continuous torus $\bb T=[0,1)$ rather than  on the interval $[0,1]$.  Moreover, we extend (on the time parameter) the functions above from $[0,T]$ to some open interval $(c,d)$ containing $[0,T]$ by imposing that the extension is constant on  $(c,0]$ and $[T,d)$, that is, given $f\colon [c,d]\times \bb T\to \bb R$, its extension  $\overline{f}\colon [c,d]\times \bb T\to \bb R$ will be defined by 
 \begin{align*}
 \overline{f}(t,u)\;=\; \begin{cases}
 f(t,u)\,,& \text{ if } (t,u) \in [0,T]\times \bb T\,,\\
  f(0,u)\,,& \text{ if } (t,u) \in (c,0)\times \bb T\,,\\
    f(T,u)\,,& \text{ if } (t,u) \in (T,d)\times \bb T\,.
 \end{cases}
 \end{align*}
 
Abusing of notation, let $\iota_\delta$ and $\iota_\eps$ be smooth approximations of the identity on $\bb T$ and $(a,b)$, respectively.    Let  $H_{\rho^{\eps,\delta}}:=h'_a(\rho^{\eps,\delta})$ where $\rho^{\eps,\delta}$ is a convolution in space and in time  (on the parameters $\eps$ and $\delta$, respectively) of the function $\rho$, that is, 
     \begin{align*}
     \rho^{\eps,\delta} (u,t) \;\coloneqq\;  \big(\rho*\iota_\eps*\iota_\delta\big)(u,t)\;=\; \int_{(a,b)}\int_{\bb T} \rho(s,v) \iota_{\eps}(u-v)\iota_{\delta}(t-s)dvds\,.
     \end{align*}
         Note now that
\begin{align*}
 \sup_{H\in \C} \Big\{L_H(\rho) + B_H(\rho) \Big\} & \;\geq \; L_{H^{\eps,\delta}_\rho}(\rho) + B_{H^{\eps,\delta}_\rho}(\rho) \\
 &\;=\;  L_{H_\rho^{\eps,\delta}}(\rho^{\eps,\delta})+\Big\{ L_{H^{\eps,\delta}_\rho}(\rho) - L_{H_\rho^{\eps,\delta}}(\rho^{\eps,\delta})\Big\}+ B_{H^{\eps,\delta}_\rho}(\rho)\,.
\end{align*}              
At this point we must  handle each of the parcels above.
By the chain rule and Fubini's Theorem, 
\begin{align*}
L_{H_{\rho^{\eps,\delta}}}(\rho^{\eps,\delta}) & \;=\; \int_0^T \<\p_s\rho^{\eps,\delta}, H_{\rho^{\eps,\delta}}\> ds  \;=\; \int_0^T \int_{\bb T}\p_s\rho^{\eps,\delta}_s(u)h'_a\big(\rho^{\eps,\delta}_s(u)\big) du ds \\
& \;=\;\int_{\bb T} \int_0^T \p_s\Big(h_a\big(\rho^{\eps,\delta}_s(u)\big)\Big) ds du\;=\; \int_{\bb T} \Big\{h_a\big(\rho^{\eps,\delta}_T(u)\big)- h_a\big(\rho^{\eps,\delta}_0(u)\big)\Big\}du
\end{align*}
and from $-\log 2\leq h_a(\cdot)\leq \log 4$ we infer that  
\begin{align}\label{fact_1}
L_{H_{\rho^{\eps,\delta}}}(\rho^{\eps,\delta}) \;\geq \; -(\log 2+\log 4)\;=\; -3\log2  \,.
\end{align}
By the same arguments of  \cite[Lemma 4.4]{LandimTsunoda},  for any fixed  $\eps>0$,
\begin{align}\label{fact_2}
\lim_{\delta\searrow 0} \Big\{ L_{H_{\rho^{\eps,\delta}}}(\rho) - L_{H_{\rho^{\eps,\delta}}}(\rho^{\eps,\delta})\Big\}\;=\;0\,.
\end{align}
Finally, $B_{H^{\eps,\delta}_\rho}(\rho)$ converges, as $\eps$ and $\delta$ decrease to zero, to 
\begin{align*}
B_{H_\rho}(\rho) \;=\; &\int_0^T\big\<\p_u\rho, \p_u h_a'(\rho)\big\>ds - \int_0^T\big\<\chi(\rho), \big(\p_u h_a'(\rho)\big)^2\big\>ds\\
\;\geq\;&  \int_0^T\Big\<\p_u\rho, \frac{(1+2a)\p_u \rho}{(\rho+a)(1-\rho+a)}\Big\>ds - \int_0^T\big\<\frac{1}{4}, \frac{(1+2a)^2(\p_u \rho)^2}{(\rho+a)^2(1-\rho+a)^2}\big\>ds\,.
\end{align*}
Taking the $\liminf$ as $a\searrow 0$, applying Fatou's Lemma and recalling \eqref{fact_1} and \eqref{fact_2}, we are lead to 
\begin{align*}
\I(\pi|\gamma) \;\geq \; -3\log 2 +\frac{3}{4}\int_0^T\int_0^1 \frac{(\p_u\rho_t(u)\big)^2}{\chi(\rho_t(u))}dudt
\end{align*}
finishing the proof.
\end{proof}
\begin{lemma}\label{lemma69}
The density $\rho$ of a trajectory $\pi\in \DMO$ is the weak solution of hydrodynamic equation  \eqref{hydroeq_Neumann} with initial condition $\gamma$ if, and only if, $\I(\pi|\gamma)=0$. Moreover, in such  case we have that
\begin{align}\label{cota}
\int_0^T\int_0^1 \frac{(\p_u\rho_t(u)\big)^2}{\chi(\rho_t(u))}dudt\;<\;\infty\,.
\end{align}
\end{lemma}
\begin{proof}
Suppose that the density $\rho$ of a trajectory $\pi\in \DMO$ is the weak solution of hydrodynamic equation  \eqref{hydroeq_Neumann} with initial condition $\gamma$.  Then, for $H\in C^{1,2}$,
\begin{align*}
J_H(\rho)\;=\; -\int_0^T \< \chi(\rho_s), (\p_u H_s)^2 \>\,ds\;\leq \;0\,.
\end{align*}
Moreover, since $\rho$ is the weak solution of \eqref{hydroeq_Neumann}, it is easy to check that the total mass of $\pi_t(du)=\rho_t du$ is conserved in time, that is, $\pi\in \mc F^\theta$, see \eqref{Ftheta}. This implies that $\I(\pi|\gamma)=0$.

Suppose now that $\I(\pi|\gamma)=0$. Therefore $J_{\eps H}(\rho)\leq 0$ for any $H\in C^{1,2}$ , which in its turn implies that the derivative of $J_{\eps H}(\rho)\leq 0$ with respect to $\eps$ is zero at $\eps=0$. This permits to conclude that the density $\rho$  is the weak solution of hydrodynamic equation  \eqref{hydroeq_Neumann} with initial condition $\gamma$.

Finally, if $\I(\pi|\gamma)<\infty$, then \eqref{cota} holds by Proposition~\ref{prop68}.
\end{proof}

Let $\Pi_1$ be the set of all paths $\pi(t, du) = \rho(t, u) du$ in $\DMO$ whose density $\rho$ is a weak solution of the
Cauchy problem (2.2) on some time interval $[0, \delta]$,  with $\delta > 0$.
\begin{lemma}
The set $\Pi_1$ is $\I$-dense.
\end{lemma}
\begin{proof}
The proof here follows the same steps of \cite[Lemma 5.3]{LandimTsunoda}.
Fix $\pi_t=\rho(t,u)du\in \DMO$ such that $\I(\pi|\gamma)<\infty$.  Let $\lambda$ be the solution of the hydrodynamic equation \eqref{hydroeq_Neumann} with $\rho_0=\gamma$. For $\delta>0$, let $\pi_t^\delta=\rho^\delta(t,u)du$ where $\rho^\delta$ evolves as $\lambda$ on the time interval $[0,\delta]$, then evolves as $\lambda$ reversed in time on $[\delta, 2\delta]$ and then evolves as $\rho$ in the remaining time interval, that is,
\begin{equation}\label{rhodelta}
\rho^\delta(t,u)\;=\;
\begin{cases}
\lambda(t,u) & \text{ if } t\in [0,\delta]\,,\\
\lambda(2\delta -t,u) & \text{ if } t\in  [\delta, 2\delta]\,,\\
\rho(t-2\delta, u) & \text{ if } t\in [2\delta, T]\,.
\end{cases}
\end{equation}
Since $\pi^\delta$ converges to $\pi$ in $\DM$ as $\delta\downarrow 0$ and $\pi^\delta\in \Pi_1$, it only remains to show that $\I(\pi^\delta|\gamma)$ converges to $\I(\pi|\gamma)$  as $\delta\downarrow 0$.
By the lower semi-continuity of the rate function, we have
$\I(\pi|\gamma)\leq\liminf_{\delta\to 0}\I(\pi^\delta|\gamma)$
hence it is missing to assure that
\begin{equation}\label{sup}
\I(\pi|\gamma)\;\geq\;\limsup_{\delta\to 0}\I(\pi^\delta|\gamma)\,.
\end{equation} 
To do so, note that 
\begin{align*}
\mc E_H(\pi^\delta)\;\leq\; 2 \mc E_H(\lambda) + \mc E_H(\pi)\;<\;\infty\,,
\end{align*}
where the last inequality above is due to the assumption $\I(\pi|\gamma)<\infty$ and Lemma~\ref{lemma69}. Using this and the fact the profile $\rho^\delta$ conserves the total mass we can infer that $\I(\pi^\delta|\gamma)<\infty$ for any $\delta$.

By linearity of integrals, we will analyze separately  the contributions on $\I(\pi^\delta|\gamma)$ from the three time intervals of \eqref{rhodelta}.
 The contribution of $[0,\delta]$ is zero by Lemma~\ref{lemma69}. 

Since the Neumann boundary conditions are invariant by a time inversion,   the profile $\rho^\delta$ is a weak solution on the time interval $[\delta, 2\delta]$ of 
\begin{equation*}
\begin{cases}
\p_t \rho(t,u)= -\p_u^2 \rho(t,u)\\
\p_u \rho(t,0) =\p_u \rho(t,1)=0
\end{cases}
\end{equation*}
which allows to conclude that the second contribution is given by
\begin{align}
\sup_{H\in C^{1,2}} \bigg\{\int_0^\delta \bigg(2\<\p_u \lambda_t, \p_u H\>- \<\chi(\lambda_t), (\p_u H)^2\>\bigg)dt \bigg\}\,.
\end{align}
Multiplying and diving the leftmost term inside parenthesis of last expression by $\sqrt{\chi(\lambda_t)}$ and applying Young's inequality $ab\leq a^2/2+b^2/2$, we can bound the previous expression from above by 
\begin{align*}
\int_0^\delta\int_0^1 \frac{(\p_u\lambda_t(u)\big)^2}{\chi(\lambda_t(u))}dudt
\end{align*}
which goes to zero as $\delta\searrow 0$ by Lemma~\ref{lemma69} and Dominated Convergence Theorem.

Finally, the third contribution is bounded above by $\I(\pi|\gamma)$ since $\pi^\delta$  on this interval is a time translation of $\pi$. Putting these things together leads to \eqref{sup} and hence finishes the proof.

\end{proof}
Next, we present the  sets $\Pi_2$, $\Pi_3$ and $\Pi_4$. 
Let $\Pi_2$ be the set of all paths $\pi(t, du) = \rho(t, u) du$ in $\Pi_1$ with the property that for every $\delta > 0$ there exists $\eps > 0$
such that $\eps \leq \rho(t, u) \leq 1 - \eps$ for all $(t, u) \in [\delta, T] \times [0,1]$.
Let $\Pi_3$ be the set of all paths $\pi(t, du) = \rho(t, u) du$ in $\Pi_2$ whose density $\rho(t,u)\,du$ belongs to the space $C^\infty[0,1]$ for any $t\in[0,T]$.
Let $\Pi_4$ be the set of all paths $\pi(t, du) = \rho(t, u) du$ in $\Pi_3$ whose density $\rho(t, u)\,du$ belongs to the space $C^{\infty,\infty}([0,T]\times [0,1])$.
\begin{lemma}\label{lemma5.456}
The sets $\Pi_2$, $\Pi_3$ and $\Pi_4$ are $\I$-dense.
\end{lemma}
The proof of the  Lemma \ref{lemma5.456} can be promptly adapted from \cite[Lemmas 5.4, 5.5 and 5.6]{LandimTsunoda},  and for this reason its proof is omitted. We thus conclude the proof of the $\I$-density, that is, the proof of Theorem~\ref{Idensity}.
\appendix
\section{Auxiliary results}\label{Appendix}
We argue here why the correct formula for $\frac{\dradon \bb P}{\dradon \overline{\bb P}}\big|_{\mc F_t}$ is that one in \eqref{der_rN}, which is the inverse of that one in 
\cite[formula (2.6), page 320]{kl}. First we note that   the Radon-Nikodym derivative  can be characterized as the unique measurable function  $\frac{\dradon \bb P}{\dradon \overline{\bb P}}\big|_{\mc F_t}$ such that 
\begin{align*}
\bb E_x[F]\;=\; \overline{\bb E}_x\Big[F\,\pfrac{\dradon \bb P}{\dradon \overline{\bb P}}\big|_{\mc F_t} \Big]\,,\quad\text {for all  bounded measurable functions } F.
\end{align*} 
As one can see in \cite[page 321]{kl}, when ending of the proof of Proposition 2.6 there, it is obtained that 
\begin{align*}
\bb E_x[F]\;=\; \overline{\bb E}_x\Bigg[&F\,  \exp\,\Big\{-\int_0^t[\lambda(X_s)-\overline{\lambda}(X_s)] \,ds +\sum_{s\leq t}\log \frac{\lambda(X_{s^-})p(X_{s^-}, X_s)}{\overline{\lambda}(X_{s^-})\overline{p}(X_{s^-}, X_s)}\Bigg]\,.
\end{align*}
Thus, 
\begin{equation*}
\frac{\dradon \bb P}{\dradon \overline{\bb P}}\Bigg|_{\mc F_t}\;=\;
\exp\Bigg\{-\Bigg(\int_0^t\big[\lambda(X_s)-\overline{\lambda}(X_s)\big]ds-\sum_{s\leq t}\log\frac{\lambda(X_{s^-})p(X_{s^-},X_s))}{\overline{\lambda}(X_{s^-})\overline{p}(X_{s^-},X_s)}\Bigg)\Bigg\}\,,
\end{equation*} 
justifying why \eqref{der_rN} is the correct formula instead of \cite[formula (2.6), page 320]{kl}.

\section*{Acknowledgements}
T.\ F. was supported  by the National Council for Scientific and Technological Development (CNPq-Brazil) through a \textit{Bolsa de Produtividade} number 301269/2018-1. P.G. thanks  FCT/Portugal for support through the project 
UID/MAT/04459/2013.  This project has received funding from the European Research Council (ERC) under  the European Union's Horizon 2020 research and innovative programme (grant agreement   n. 715734).

\bibliographystyle{abbrv}
\bibliography{bibliography}

\end{document}